\documentclass[a4paper,11pt]{amsart}

\usepackage{amstext,amsfonts,amsthm,graphicx,amssymb,amscd,epsfig}
\usepackage{amsmath}
\usepackage[ansinew]{inputenc}    

\usepackage{mathrsfs}

\theoremstyle{definition}
\newtheorem{definition}{Definition}[section]
\newtheorem{ex}[definition]{Example}
\newtheorem{rem}[definition]{Remark}

\theoremstyle{plain}
\newtheorem{prop}[definition]{Proposition}
\newtheorem{lem}[definition]{Lemma}
\newtheorem{coro}[definition]{Corollary}
\newtheorem{teo}[definition]{Theorem}

\newfont{\bbb}{msbm10 scaled\magstephalf}     
\def\C{\mathbb C}
\def\K{\mathbb K}
\def\R{\mathbb R}

\def\cod{\operatorname{cod}}

\def\R{\mbox{\bbb R}}

\def\O{\mathcal O}
\def\A{\mathscr A}
\def\Lift{\operatorname{Lift}}
\def\Derlog{\operatorname{Derlog}}
\def\Aecod{\operatorname{\mbox{$\A_e$-cod}}}
\def\Acod{\operatorname{\mbox{$\A$-cod}}}
\def\Kecod{\operatorname{\mbox{$\mathscr K_e$-cod}}}
\def\Kcod{\operatorname{\mbox{$\mathscr K$-cod}}}

\title{The extra-nice dimensions}

\author{R. Oset Sinha, M. A. S. Ruas, R. Wik Atique}

\date{}

\address{Departament de Matem\`atiques,
Universitat de Val\`encia, Campus de Burjassot, 46100 Burjassot,
Spain}

\email{raul.oset@uv.es}

\address{Instituto de Ci\^encias Matem\'aticas e de Computa\c{c}\~ao - USP,
Av. Trabalhador s\~ao-carlense, 400 - Centro,
CEP: 13566-590 - S\~ao Carlos - SP, Brazil}

\email{maasruas@icmc.usp.br}
\email{rwik@icmc.usp.br}

\thanks{The first author is partially supported by DGICYT Grant MTM2015--64013--P. The second author is partially supported by FAPESP grant no. 2014/00304-2
and CNPq grant no. 306306/2015-8.  The
third author is partially supported by FAPESP grant no.
2015/04409-6.}

\begin{document}
\begin{abstract}
We define the extra-nice dimensions and prove that the subset of locally stable 1-parameter families in $C^{\infty}(N\times[0,1],P)$, also known as pseudo-isotopies, is dense if and only if the pair of dimensions $(\dim N, \dim P)$ is in the extra-nice dimensions. This result is parallel to Mather's characterization of the nice dimensions as the pairs $(n,p)$ for which stable maps are dense.
The extra-nice dimensions are characterized by the property that discriminants of stable germs in one dimension higher have $\mathcal A_e$-codimension 1 hyperplane sections. They are also related to the simplicity of $\mathcal A_e$-codimension 2 germs. We give a sufficient condition for any
$\mathscr A_e$-codimension 2 germ to be simple and give an example
of a corank 2 codimension 2 germ in the nice dimensions which is not simple.
Then we establish the boundary of the extra-nice dimensions. Finally we answer a question posed by Wall about the codimension of
non-simple maps.
\end{abstract}

\maketitle

\section{Introduction}

Around the middle of last century Whitney formulated the concept of stability of smooth maps  and characterized stable singularities in dimensions $(n,p)$ with
$p \geq 2n$, $(n,2n-1)$, $(2,2)$ showing that in these cases stable maps are dense in the space of $C^{\infty}$ maps. He then conjectured that this holds in any pair $(n,p)$. Thom showed that this is not the case (see \cite{ThomLevine}) by giving an example in  $(9,9)$ of a singularity which appears generically in a 1-parameter family of maps. This singularity  has  $\mathscr A_e$-codimension 1 and is not simple. A germ is simple if there are only a finite number of orbits nearby, therefore, in the pair of dimensions $(9,9)$ not all maps can be approximated by stable maps and so the stable maps are not dense. He then conjectured that topologically stable maps are always dense and this was proved by Mather (\cite{Mac0dense}).

In his well known series of papers about stability of $C^{\infty}$ maps Mather showed that the set of stable maps $f: N^n \to P^p$ is dense in $C_{pr}^{\infty}(N,P)$ (proper $C^{\infty}$ maps) with the Whitney strong topology if and only if the pair of dimensions $(n,p)$ is in the {\em nice dimensions} (\cite{MaV}), which he determined completely in \cite{MaVI}. Mather gave a stratification of the set $J^k(n,p)$ of $k$-jets of smooth maps by $\mathscr K$-orbits and characterized stability in terms of transversality of the jet extension of the map to this stratification.
More precisely, he defined the nice dimensions as the pairs $(n,p)$ such that there exists a Zariski closed $\mathscr K$-invariant set $\Pi(n,p)$ in $J^k(n,p),$ for sufficiently high $k,$ of codimension bigger than $n$ such that its complement in $J^k(n,p)$ is the union of finitely many $\mathscr K$-orbits.

When the pair $(n,p)$ is in the nice dimensions and the source manifold $N$ is compact, an important problem in the applications of singularity theory
to differential topology is the characterization of the simplest existing paths between two stable maps. A 1-parameter family
$F: N\times [0,1] \to P$ connecting two non-equivalent stable maps always intersects the set of non stable maps for a finite number of values of the parameter, the bifurcation points.
The classification of  singularities of bifurcation points in generic families of maps is a fundamental step in results on elimination of singularities, which is still an active field of research
(\cite{levine}, \cite{behrenshayano}), and on various results about the topology of the space of smooth maps such as work by J. Cerf (\cite{cerf}) or K. Igusa (\cite{igusa}) or even Vassiliev's theory of topological invariants (\cite{vassiliev}).

We say that a family $F: N\times [0,1] \to P$ is a {\it locally stable family } if $F_t: N \to P$ is a stable map for all $t \in \mathbb [0,1]$ except possibly for a finite number of values $\{t_1,\ldots, t_k\}$ and  the non-stable singularities of $F_{t_i}$ are a finite number of points $x_j \in N,$ and the map $(F, t) : N\times [0,1] \to P\times [0,1]$ is a locally $\mathcal A_e$-versal unfolding of $F_{t_i}$ for all non-stable points $x_j.$
This definition implies that the non-stable singularities of $F_{t_i}$ are $\mathcal A$-finitely determined and their $\mathcal A_e$-codimension is  equal to 1.

In this paper we obtain a parallel result to Mather's characterization of the nice dimensions. First we define the {\em extra-nice dimensions} as the pairs $(n,p)$ where there exists a smallest Zariski closed $\mathscr A$-invariant set $\Lambda(n,p)$ in $J^k(n,p),$ for sufficiently high $k,$ of codimension greater than $n+1$ whose complement is a finite number of $\mathscr A$-orbits. Then we prove that the subset of stable 1-parameter families in $C^{\infty}(N\times [0,1],P)$ is dense if and only if the pair $(n,p)$ is in the extra-nice dimensions.

 In the nice dimensions all the $\mathscr A_e$-codimension $1$ singularities are simple (see Proposition \ref{cod1simple}). Goryunov (\cite{goryunov}), Cooper, Mond and Wik Atique (\cite{robertamond}) and Houston (\cite{houstonclass}) studied the classification of germs and multigerms of $\mathscr A_e$-codimension $1$, corank $1.$
A recent work by Oset Sinha, Ruas and Wik Atique (\cite{ORW1}) defined operations that allow the classification of $\mathscr A_e$-codimension 2 multigerms in the nice dimensions and a natural question arises: are all of these simple? In Section 3 we prove that all corank $1$ $\mathscr A_e$-codimension $2$ monogerms in $(n,p)$ are simple when $(n+1,p+1)$ is in the
nice dimensions. We give a sufficient condition for any $\mathscr A_e$-codimension $2$ germ to be simple. This condition is related to stable germs in one
dimension higher having $\mathscr A_e$-codimension $1$ hyperplane
sections. We also give an example
of a corank $2$ codimension $2$ germ in the nice dimensions which is not
simple.

In Section 4 we define the extra-nice dimensions, we relate them to the simplicity of $\mathscr A_e$-codimension 2 germs and we characterize them by stable germs in one
dimension higher having $\mathscr A_e$-codimension 1 hyperplane
sections (the sufficient condition in Section 3). In Section 5 we determine the boundary of the extra-nice dimensions completely. Figure 1 shows this boundary and compares it to the boundary of the nice dimensions. In Section 6 we characterize the extra-nice dimensions by the density of locally stable 1-parameter families.

Section 2 establishes the necessary notation and basic results. Finally, in Section 7, we answer a question posed by Wall about the codimension of
non-simple maps. We define further refinements of the nice dimensions and give an example in the equidimensional case of a stratification in terms of increasing codimension of the subset of non-simple maps.

\emph{Acknowledgements:} The authors thank J. J. Nuño-Ballesteros for spotting errors in previous versions of the manuscript and T. Nishimura for helpful discussions, both helping to improve the presentation of the results.
	
\section{Notation}\label{section-notation}

We consider map-germs $f:(\K^n,S)\to (\K^p,0)$, where $\K=\R$ or $\K=\C$, and $S\subset \K^n$ a finite subset. For simplicity, we will say that $f$ is smooth if it is smooth (i.e. $C^\infty$) when $\K=\R$ or holomorphic when $\K=\C$. We denote by $\O_n=\O_{\K^n,S}$ and $\O_p=\O_{\K^p,0}$ the rings of smooth function germs in the source and target respectively, by $\mathscr M_n$ and $\mathscr M_p$ the maximal ideals of $\O_n$ and $\O_p$ respectively and by $\theta_n=\theta_{\K^n,S}$ and $\theta_p=\theta_{\K^p,0}$ the corresponding modules of vector field germs. The module of vector fields along $f$ will be denoted by $\theta(f)$. Associated with $\theta(f)$ we have two morphisms $tf:\theta_n\rightarrow \theta(f)$, given by $tf(\chi)=df\circ\chi$,
and $wf:\theta_p\rightarrow \theta(f)$, given by $wf(\eta)=\eta\circ f$. Let $f^*:\O_p\to\O_n$ be the induced map of $f$ given by composition with $f$ on the right.
Let $\mathscr G=\A_e,\A,\mathscr K_e,\mathscr K$. The $\mathscr G$-tangent space and the $\mathscr G$-codimension of $f$ are defined respectively as
$$T\A_e f=tf(\theta_n)+wf(\theta_p),\quad
\Aecod(f)=\dim_\K\frac{\theta(f)}{T\A_e f},$$
$$T\A f=tf(\mathscr M_n\theta_n)+wf(\mathscr M_p\theta_p),\quad
\Acod(f)=\dim_\K\frac{\mathscr M_n\theta(f)}{T\A f},$$
$$T\mathscr K_e f=tf(\theta_n)+f^*\mathscr M_p\theta(f),\quad
\Kecod(f)=\dim_\K\frac{\theta(f)}{T\mathscr K_e f},$$
$$T\mathscr K f=tf(\mathscr M_n\theta_n)+f^*\mathscr M_p\theta(f),\quad
\Kcod(f)=\dim_\K\frac{\mathscr M_n\theta(f)}{T\mathscr K f}.$$
It follows from Mather's infinitesimal stability criterion \cite{MaII} that a germ is stable if and only if its $\A_e$-codimension is 0. We refer to Wall's survey paper \cite{Wall} and Nu\~no-Ballesteros and Mond book \cite{nunomond} for general background on the theory of
singularities of mappings.

There are some relations between the different codimensions. One between the $\A$-codimension and the $\A_e$-codimension
due to L. Wilson~\cite{wilson} (a proof can be found in \cite{riegerwilson}):
$$\Aecod(f)=\Acod(f)+r(p-n)-p,$$ if $f$ has finite $\mathscr A_e$-codimension and is not stable, where $r=|S|$ is the number of branches. And another one between the $\mathscr K$-codimension and the $\mathscr K_e$-codimension, which can be found in \cite{Wall}:
$$\Kecod(f)=\Kcod(f)+r(p-n).$$

We say that $f:(\K^n,0)\to (\K^p,0)$ has corank $r$ if $df(0)$ has rank $\min{n,p}-r$.

We say that $f$ has finite singularity type or it is $\mathscr K$-finite if
$\Kecod(f)<\infty$. Another remarkable result of Mather is that $f$ has finite singularity type if and only if it admits an $s$-parameter stable unfolding (see \cite{Wall}). We recall that an $s$-parameter unfolding of $f$ is another map-germ
$$F:(\K^n\times\K^s,S\times\{0\})\to (\K^p\times\K^s,0)
$$
of the form $F(x,\lambda)=(f_\lambda(x),\lambda)$ and such that $f_0=f$.

Along the paper, we use the notation of small letters $x_1,\dots,x_n,\lambda_1,\dots,\lambda_r$ for the coordinates in $\K^n\times\K^r$ and capital letters $X_1,\dots,X_p,\Lambda_1,\dots,\Lambda_r$ for the coordinates in $\K^p\times\K^r$.

A multigerm $f=\{f_1,\ldots,f_r\}:(\K^n,S)\rightarrow(\K^p,0)$ with
$S=\{x_1,\ldots,x_r\}$ is simple if there exists a finite number of
$\mathscr A$-classes such that for every
unfolding
$F:(\K^n\times\K^s,S\times\{0\})\rightarrow(\K^p\times\K^s,0)$ with
$F(x,\lambda)=(f_{\lambda}(x),\lambda)$ and $f_0=f$ there exists a
sufficiently small neighbourhood $U$ of $S\times\{0\}$ such that for
every $(y_1,\lambda),\ldots,(y_r,\lambda)\in U$ where
$F(y_1,\lambda)=\ldots=F(y_r,\lambda)$ the multigerm
$f_{\lambda}:(\K^n,\{y_1,\ldots,y_r\})\rightarrow(\K^p,f_{\lambda}(y_i))$
lies in one of those finite classes.

\begin{definition} Let $f:(\K^n,S)\to (\K^p,0)$ be a smooth map-germ.
A vector field germ $\eta\in \theta_p$ is called \textit{liftable
over $f$}, if there exists $\xi\in\theta_n$ such that
$df\circ\xi=\eta\circ f$ (i.e., $tf(\xi)=wf(\eta)$). The set of vector
field germs liftable over $f$ is denoted by $\Lift(f)$ and is an
$\mathcal{O}_p$-submodule of $\theta_p$.

\end{definition}


When $\K=\C$ and $f$ has finite singularity type, we always have the inclusion $\Lift(f)\subseteq \Derlog(\Delta(f))$,  where $\Delta(f)$ is the discriminant of $f$ (i.e., the image of non submersive points of $f$) and $\Derlog(\Delta(f))$ is the submodule of $\theta_p$ of vector fields which are tangent to $\Delta(f)$. Moreover, we have the equality $\Lift(f)=\Derlog(\Delta(f))$ in case $f$ has finite $\A_e$-codimension (see \cite{damon,nunomond}).

\begin{definition}\label{aug}
Let $h:(\mathbb{K}^n,S)\rightarrow (\mathbb K^p,0)$ be a map-germ
with a 1-parameter stable unfolding $H(x,\lambda)=(h_{\lambda}(x),\lambda)$. Let $g:(\mathbb K^q,0)\rightarrow (\mathbb K,0)$ be a function-germ. Then, the \textit{augmentation of h by H and g}
is the map $A_{H,g}(h)$ given by $(x,z)\mapsto (h_{g(z)}(x),z)$. A map which is not an augmentation is called primitive.
\end{definition}

In Theorem 4.4 in \cite{houston} it is proved that $$\Aecod(A_{H,g}(h))\geq\Aecod(h)\tau(g)$$ where $\tau$ is the Tjurina number of the function $g$. Equality is reached if $g$ is quasihomogeneous or $H$ is a substantial unfolding (see \cite{houston}). In this paper we will only use the particular case when $g$ is a Morse function and so $\Aecod(A_{H,g}(h))=\Aecod(h)$.



Following Mather in \cite{MaIV} if $f:(\K^n,0)\to (\K^p,0)$ has finite singularity type then there is a stable germ $F:(\mathbb{K}^{n+s},0)\rightarrow (\mathbb K^{p+s},0)$ and a germ of an immersion $i:(\mathbb{K}^n,0)\rightarrow (\mathbb K^{p+s},0),$ $i\pitchfork F, $  such that  $f$  is the pull-back of $F$ by $i$ in the diagram:

\begin{equation*}
\begin{CD}
(\mathbb{K}^{n+s},0) @>{F}>> (\mathbb K^{p+s},0)  \\ @AA{}A        @A{i}AA   \\
(\K^n,0) @>{f}>> (\K^p,0)
\end{CD}
\end{equation*}

Any germ $f$ is a pull-back of a stable $s$-parameter unfolding $F$
by the natural inclusion $i.$  Damon (\cite{damon}, for $\K=\mathbb{C}$) and Houston (\cite{houston}, for $\K=\mathbb{R}$) proved that
$\Aecod(f)=\mathscr K_{\Delta(F),e}$-cod$(i)$, where
$$\mathscr K_{\Delta(F),e}\mbox{-cod}(i)=\dim_{\mathbb K} N\mathscr K_{\Delta(F),e}(i)=\dim_{\mathbb K}\frac{\theta(i)}{ti(\theta_p)+i^*\Lift(F)}.$$

Furthermore, if $L:(\mathbb K^{p+s},0)\longrightarrow (\mathbb K^s,0)$ is such that
 $L\circ i=0$, then $\Aecod(f)= _{\Delta(F)}\mathscr K_e$-$\cod(L)$, where $$_{\Delta(F)}\mathscr K_e\mbox{-cod}(L)=\dim_{\mathbb K} N_{\Delta(F)}\mathscr K_e(L)=\dim_{\mathbb K}\frac{\theta(L)}{tL(\Lift(F))+L^*\mathscr M_s\theta(L)}.$$

 When $s=1$ we say that $L$ defines the  hyperplane section $f$ of $F$. Besides, to obtain a hyperplane section of $\A_e$-codimension 1 it is sufficient to consider the 1-jet of $L$ (see for example \cite{mondwa}). Also notice that if $F$ is minimal, $_{\Delta(F)}\mathscr K_e$-$\cod(L)=1$ if and only if $\mathscr M_{p+1}\theta(L)\subset tL(\Lift(F))+L^*\mathscr M_s\theta(L)$.

\section{On simplicity of codimension 2 germs in the nice dimensions}

We need the following characterisation of the openness of an $\A$-orbit in the $\mathscr K$-orbit.

\begin{teo}(\cite{riegerruasmdef},\cite{ruas})\label{open}
Let $f:(\K^n,0)\rightarrow(\K^p,0)$ be a $\mathscr K$-finite germ
and let $\{v_1,\ldots,v_r\}$ be a basis for
$$N:=\frac{\theta(f)}{T\A_ef+f^*\mathscr M_p\theta(f)}.$$ The
$\A$-orbit of $f$ is open in the $\mathscr K$-orbit if and only if
$f_iv_j\in T\A f$, $i=1,\ldots,p$ and $j=1,\ldots,r$ (mod
$f^*\mathscr M_p^2 \theta(f)$).
\end{teo}

If $f$ is stable, the $\A$-orbit is open in the $\mathscr K$-orbit.

The $\A$-orbit is open in the $\mathscr K$-orbit if and only if
$T\A f=T\mathscr K f$, so $\Acod(f)=\Kcod(f)$. By the formulas
relating the $\mathscr G$-codimension to the $\mathscr
G_e$-codimension in the previous section, this is equivalent to
$\Aecod(f)=\Kecod(f)-p$, so basically, a non stable germ has the
$\A$-orbit open in the $\mathscr K$-orbit if and only if there is no
stable germ in the $\mathscr K$-orbit before the versal unfolding.

\begin{ex}
The germ $(x^5+yx,y)$ has $\mathscr A_e$-codimension 3 but admits a 2-parameter stable unfolding, so its $\A$-orbit is not open in the $\mathscr K$-orbit. However, $(x^5+yx+x^7,y)$ has $\mathscr A_e$-codimension 2 and its $\A$-orbit is open in the $\mathscr K$-orbit (see \cite{rieger}).
\end{ex}

In Subsection \ref{cod2notsimple} we will see an example of a $\mathscr K$-orbit which does not admit an open $\mathscr A$-orbit.

The germs of $\mathscr A_e$-codimension 1 and corank 1 in the nice dimensions are well known and are hyperplane linear sections of stable germs. We review them here for sake of completeness. The case of hyperplane sections of stable corank 2 germs in $(n,n+1)$ has been studied in \cite{tesismirna}.

\begin{prop}\label{iestrella}
Let $F:(\mathbb K^{n+1},0)\longrightarrow
(\mathbb K^{p+1},0)$ be a stable  corank 1 germ, $(n,p)$ nice dimensions.
Then there exists $f:(\mathbb K^n,0)\longrightarrow (\mathbb
K^p,0)$ obtained by a hyperplane section of $F$ such that $\mathscr
A_e$-$\cod(f)=1$.
\end{prop}
\begin{proof}

Suppose first that $F$ is minimal.

1) Case $n\geq p$. Let $g_0:(\mathbb K^m,0)\rightarrow(\mathbb K,0)$ be a simple function singularity of type $A_k$, $D_k$, $E_6$, $E_7$ or $E_8$, and let $\varphi_1=1,\varphi_2,\ldots,\varphi_{\mu}$ be a homogeneous basis for $\frac{\mathcal O_n}{Jg_0}$ where $\varphi_{\mu}$ is the unique highest weight term. Then, by \cite{MaIV}, the map germs $G:(\K^m\times\K^{\mu-1},0)\rightarrow(\K\times\K^{\mu-1},0)$ given by $$G(x,u_2,\ldots,u_{\mu})=(g_0(x)+\sum_{i=2}^{\mu}u_i\varphi_i,u_2,\ldots,u_{\mu})$$ are stable minimal corank 1 germs. Moreover, any stable minimal corank 1 germ is $\mathscr A$-equivalent to one of such germs. The section $u_{\mu}=0$ defines an $\mathscr A_e$-codimension 1 section, see~\cite{goryunov}.

2) Case $n<p$. By \cite{MaIV}, any stable minimal corank 1 germ is $\mathscr A$-equivalent to $G:(\K^n\times\K,0)\rightarrow(\K^p\times\K,0)$ given by
\begin{align*}
G(u_1,\ldots,u_{l-1},v_1,\ldots,v_{l-1},w_{11},w_{12},\ldots,w_{rl},y,\lambda)=(u_1,\ldots,u_{l-1},v_1,\ldots,v_{l-1},\\
w_{11},w_{12},\ldots,w_{rl},y^{l+1}+\sum_{i=1}^{l-1}u_iy^i,y^{l+2}+\sum_{i=1}^{l-1}v_iy^i+\lambda y^l,\sum_{i=1}^lw_{1i},\ldots,\sum_{i=1}^lw_{ri},\lambda),
\end{align*}
where $r=p-n-1$ and $l$ is such that $l+1$ is the multiplicity of the germ and $n=l(r+2)-1$. In particular, when $p+1=n+2$, then $r=0$ and so $n+1=2l$ is even, which means that there is no stable minimal germ $G:(\K^n\times\K,0)\rightarrow(\K^p\times\K,0)$ when $n+1$ is odd (\cite{robertamond}). The section $\lambda=0$ defines an $\mathscr A_e$-codimension 1 section, see~\cite{houstonclass}.

Now let $F$ be $\mathscr A$-equivalent to $G.$ Then there exist germs of
diffeomorphisms $\phi:(\mathbb K^n\times \mathbb K,0)\longrightarrow
(\mathbb K^n\times\mathbb K,0)$ and $\psi:(\mathbb K^p\times \mathbb
K,0)\longrightarrow (\mathbb K^p\times\mathbb K,0)$ such that
$\psi\circ F= G\circ \phi.$ By Lemma 6.1 in \cite{NORW} $d\psi(\Lift(F))=\psi^*(\Lift(G)).$ Suppose $L$  defines a hyperplane section of $G$ of $\mathscr A_e$-codimension 1. We have
$$t(L\circ\psi)(\Lift(F))+\langle L\circ\psi\rangle\theta(L\circ\psi)=tL(\psi^*(\Lift(G)))+\langle
L\circ\psi\rangle\theta(L\circ\psi)=$$
$$\psi^*(tL(\Lift(G))+\langle
L\rangle\theta(L))=\psi^*(\mathscr M_{p+1}\theta(L))=\mathscr
M_{p+1}\theta(L\circ\psi)$$ by linearity of $L$.
Therefore $L\circ\psi$ defines a hyperplane section of $F$ of $\mathscr A_e$-codimension $1.$

If $F$ is not minimal, then $F$ is $\mathscr A$-equivalent to $Id_{\mathbb
K^r}\times F'$ where $F':(\mathbb K^{n-r},0)\longrightarrow (\mathbb
K^{p-r},0)$ is minimal. So there exists an $\mathscr A_e$-codimension 1
hyperplane section $f':(\mathbb K^{n-r-1},0)\longrightarrow
(\mathbb K^{p-r-1},0)$ of $F'$. If we augment $f'$ by the
function $\phi(x_1,\ldots,x_{r+1})=x_1^2+\ldots+x_{r+1}^2$ we obtain
a germ $f:(\mathbb K^n,0)\longrightarrow (\mathbb K^p,0)$ of
$\mathscr A_e$-codimension 1.
\end{proof}

From the normal forms showed above all corank 1 $\mathscr
A_e$-codimension 1 germs in the nice dimensions are simple. We give here a proof of this fact
for any corank.

\begin{prop}\label{cod1simple}
If a pair $(n,p)$ is in the nice dimensions then all
$\mathscr A_e$-codimension 1 germs in that pair are simple.
\end{prop}
\begin{proof}
Let $f:(\mathbb K^n,0)\longrightarrow (\mathbb
K^p,0)$ be an $\mathscr A_e$-codimension 1 germ.

Recall from~\cite{robertamond} and~\cite{houston} that $f$ is either primitive or a quadratic augmentation, that is, an augmentation of a primitive $\mathscr A_e$-codimension 1 germ  $h:(\mathbb K^{n-l},0)\longrightarrow (\mathbb
K^{p-l},0)$ by a Morse function $g$.

For germs of $\mathscr A_e$-codimension 1, the $\A$-orbit is open in its $\mathscr K$-orbit if and only if the germ is primitive (see \cite{houstoncod1} or \cite{ruas}). In fact,
if $f$ is primitive then its  miniversal unfolding $F$ is minimal.  We have
$$\dim\frac{\theta(f)}{tf(\theta_n)+wf(\theta_p)}=1$$ and therefore
$$\dim\frac{\theta(f)}{tf(\theta_n)+wf(\theta_p)+f^*\mathscr M_p\theta(f)}\leq
1.$$ This dimension cannot be 0 because $F$ is minimal, therefore
$tf(\theta_n)+wf(\theta_p)+f^*\mathscr M_p\theta(f)\subset
tf(\theta_n)+wf(\theta_p)$ and so, by Theorem \ref{open}, the $\mathscr A$-orbit of $f$  is open
in the $\mathscr K$-orbit  of $f$, i.e. $\Acod(f)=\Kcod(f)$. Suppose $f$ is non simple, then there must be a modal
stratum $Y$ with $\cod_{J^k(n,p)} (Y)\leq \Acod(f)-1=\Aecod(f)+n-1=n$. So we cannot find a subset of $J^k(n,p)$ with codimension greater than $n$ whose complement is a finite union of $\mathscr K$-orbits and this contradicts the fact that we are in the nice dimensions.

When $f$ is an augmentation then $f$ is $\mathscr A$-equivalent to $A_{H,g(z)}(h)$ where $H(x,\lambda)=(h_{\lambda}(x),\lambda)$ is the versal unfolding of the primitive germ $h$ and $g$ is a Morse function. Therefore $(x,z,u)\mapsto (h_{g(z)+u}(x),z,u)$ is a versal unfolding of $f$ and $f$ can be deformed only in a finite number of $\mathscr A$-classes.
\end{proof}


\begin{prop}\label{equiv}
Let $F:(\mathbb K^n\times \mathbb
K,0)\longrightarrow (\mathbb K^p\times\mathbb K,0)$ be a stable minimal germ, the following are equivalent:
\begin{enumerate}
\item[i)] There exists $f:(\mathbb K^n,0)\longrightarrow (\mathbb K^p,0)$ such that $F$ is a versal unfolding of $f$ and its $\mathscr A$-orbit is open in its $\mathscr K$-orbit.
\item[ii)] There exists $f$ (unfolded by $F$) which is a primitive $\mathscr A_e$-codimension 1 germ.
\item[iii)] There exists an immersion $i:(\mathbb K^p,0)\to(\mathbb K^p\times\mathbb K,0)$ such that $i^*(F)$ is an $\mathscr A_e$-codimension 1 germ.
\item[iv)] There exists a linear map $L:\mathbb K^p\times\mathbb K\longrightarrow \mathbb K$ such that $\mathscr M_{p+1}$ is contained in  $L(\Lift(F))+\langle L\rangle.$
\end{enumerate}
\end{prop}
\begin{proof}
i) if and only if ii) can be found in \cite{ruas} or \cite{houstoncod1}. Notice that since $F$ is minimal, the $\mathscr K$-codimension of $F$ (and of $f$) must be $n+1$. If the $\mathscr A$-orbit of $f$ is open in its $\mathscr K$-orbit, then $f$ has $\A$-codimension $n+1$ and thus $\mathscr A_e$-codimension 1.  ii) if and only if iii) is trivial by Damon's Theorem on transverse fibre squares and the fact that $F$ is minimal. iii) if and only if iv) follows directly from Damon's result $\Aecod(f)= _{\Delta(F)}\mathscr K_e$-$\cod(L)$ and the fact that $i(\mathbb K^p)=L^{-1}(0)$.
\end{proof}

\begin{teo}\label{cod2simple}
Let $(n+1,p+1)$ be nice dimensions. All corank 1 $\mathscr A_e$-codimension 2 germs in $(n,p)$ are simple.
\end{teo}
\begin{proof}
Suppose we have a corank $1$ germ $f:(\mathbb K^n,0)\longrightarrow (\mathbb
K^p,0)$ with $\mathscr A_e$-codimension 2 which is not simple. Let
$X$ denote the $\mathscr A$-orbit of $f$. Then $\cod_{J^k(n,p)} (X)=\Acod(f)=n+2$ and the codimension of the modal stratum
$Y$ is $\cod_{J^k(n,p)} (Y)\leq \Acod(f)-1=\Aecod(f)+n-1=n+1$.

Suppose $f$ does not admit a 1-parameter stable unfolding, then $f$
is primitive. We know that
$$\dim\frac{\theta(f)}{tf(\theta_n)+wf(\theta_p)}=2$$ and therefore
$$\dim\frac{\theta(f)}{tf(\theta_n)+wf(\theta_p)+f^*(\mathscr M_p)\theta(f)}\leq
2.$$ Since $f$ does not admit a 1-parameter
stable unfolding this dimension cannot be $1$, therefore it must be 2. This implies that
$tf(\theta_p)+wf(\theta_n)+f^*(\mathscr M_p)\theta(f)\subset
tf(\theta_p)+wf(\theta_n)$ and by Theorem \ref{open} the $\mathscr A$-orbit is
open in the $\mathscr K$-orbit, so the codimension of the $\mathscr
K$-orbit is equal to $n+2$. Since the codimension of the modal
stratum is $n+1$ we have a 1-parameter family of $\mathscr
K$-orbits of codimension $n+2$. So in $(n+1,p+1)$ we have modality of $\mathscr
K$-orbits of codimension $n+2$, which contradicts that $(n+1,p+1)$ are nice dimensions.

Now suppose $f$ admits a 1-parameter stable unfolding $F:(\mathbb
K^n\times \mathbb K,0)\longrightarrow (\mathbb
K^p\times\mathbb K,0)$. By Proposition \ref{iestrella}, there exists $f'$
in the same $\mathscr K$-orbit such that $\Aecod(f')=1$.
So $\cod_{J^k(n,p)}
(X')=n+1$, where $X'$ denotes the $\mathscr A$-orbit of $f'$. As
the codimension of the $\mathscr A$-orbit of $f'$ is greater than or equal to the
codimension of the modal stratum, then $f'$ is not simple, which
contradicts Proposition \ref{cod1simple}.

\end{proof}

Notice that in the above proof the hypothesis of corank 1 is only used to ensure the existence of $\mathscr A_e$-codimension 1 hyperplane sections so the above result can be rephrased as

\begin{prop}\label{sectionimpliescod2simple}
Let $(n+1,p+1)$ be nice dimensions. If all stable germs $F:(\mathbb
K^{n+1},0)\longrightarrow (\mathbb K^{p+1},0)$ admit a codimension 1 hyperplane section, then all codimension 2 germs $f:(\mathbb
K^{n},0)\longrightarrow (\mathbb K^{p},0)$ are simple.
\end{prop}


We believe that the results in this section  hold also for multigerms, but in this paper we are concerned only with monogerms.

\subsection{An example of a corank 2 codimension 2 germ which is not simple}\label{sectioncod2notsimple}

From Theorem \ref{cod2simple} in order to find $\mathscr A_e$-codimension 2 non simple germs we have two possibilities, either it has corank greater than 1, or it is has corank 1 and is just below the boundary of the nice dimensions. In the latter case it must come from a section of an $\mathscr A_e$-codimension 1 non simple germ in the boundary of the nice dimensions and must have $\mathscr A$-orbits open in the $\mathscr K$-orbits.

From Proposition \ref{equiv} it follows that a stable germ $F:(\mathbb
K^{n+1},S)\longrightarrow (\mathbb K^{p+1},0)$ admits an $\mathscr A_e$-codimension 1 hyperplane section if and only if there exists a linear map $L$ such that $\mathscr M_{p+1}\subset L(\Lift(F))+\langle L \rangle$.

Consider the stable germ $F_{3,3}:(\mathbb K^6,0)\to (\mathbb K^6,0)$ given by
$$F_{3,3}(x,y,u_1,u_2,u_3,u_4)=(x^3+y^3+u_1x+u_2y+u_3x^2+u_4y^2,xy,u_1,u_2,u_3,u_4)=(X,Y,U_1,...,U_4).$$

\begin{lem}\label{lift33} $\Lift(F_{3,3})$ is generated by

$$\eta_{1,2,3}=\left(\begin{array}{c}
3X\\
2Y\\
2U_1\\
2U_2\\
U_3\\
U_4\end{array}\right),
\left(\begin{array}{c}
2U_1U_2+6Y^2+4U_3U_4Y\\
X\\
-3U_2U_3-5U_4Y\\
-3U_1U_4-5U_3Y\\
-4U_2\\
-4U_1\end{array}\right),
\left(\begin{array}{c}
\frac{4}{3}U_2Y\\
-\frac{1}{9} U_3Y\\
X+\frac{1}{9}U_1U_3\\
-\frac{5}{3} U_4Y\\
-\frac{2}{3} U_1+\frac{2}{9} U_3^2\\
-2Y\end{array}\right)$$

$$\eta_{4,5,6}=\left(\begin{array}{c}
\frac{4}{3}U_1Y\\
-\frac{1}{9} U_4Y\\
-\frac{5}{3} U_3Y\\
X+\frac{1}{9}U_2U_4\\
-2Y\\
-\frac{2}{3} U_2+\frac{2}{9} U_4^2\\
\end{array}\right),
\left(\begin{array}{c}
\frac{5}{3}U_4Y^2+\frac{1}{9}U_2U_3Y\\
(-\frac{2}{9}U_1+\frac{2}{27}U_3^2)Y\\
-\frac{4}{3}U_2Y+\frac{2}{9}U_1^2-\frac{2}{27}U_1U_3^2\\
-2Y^2-\frac{2}{9}U_3U_4Y\\
X+\frac{5}{9}U_1U_3-\frac{4}{27}U_3^3\\
-\frac{1}{3}U_3Y \end{array}\right),
\left(\begin{array}{c}
\frac{5}{3}U_3Y^2+\frac{1}{9}U_1U_4Y\\
-\frac{2}{9}U_2Y+\frac{2}{27}U_4^2Y\\
-2Y^2-\frac{2}{9}U_3U_4Y\\
-\frac{4}{3}U_1Y+\frac{2}{9}U_2^2-\frac{2}{27}U_4^2U_2\\
-\frac{1}{3}U_4Y\\
X+\frac{5}{9}U_2U_4-\frac{4}{27}U_4^3\end{array}\right).$$
\end{lem}
\begin{proof}
$F_{3,3}$ is a free divisor so $\Lift(F_{3,3})$ is generated by 6 vector fields. The Euler vector field is clearly liftable and the other 5 are all linearly independent and liftable by the following lowerable vector fields:
$$\xi_2=\left(\begin{array}{c}
u_2+y^2+u_4y\\
u_1+x^2+u_3x\\
-3u_2u_3-5u_4xy\\
-3u_1u_4-5u_3xy\\
-4u_2\\
-4u_1\end{array}\right),
\xi_3=\left(\begin{array}{c}
-\frac{1}{3}x^2-\frac{1}{9}u_3x\\
\frac{1}{3}xy\\
x^3+y^3+u_1x+u_2y+u_3x^2+u_4y^2+\frac{1}{9}u_1u_3\\
-\frac{5}{3}u_4xy\\
-\frac{2}{3}u_1+\frac{2}{9}u_3^2\\
-2xy\end{array}\right)$$

$$\xi_{4}=\left(\begin{array}{c}
\frac{1}{3}xy\\
-\frac{1}{3}y^2-\frac{1}{9}u_4y\\
-\frac{5}{3}u_3xy\\
x^3+y^3+u_1x+u_2y+u_3x^2+u_4y^2+\frac{1}{9}u_2u_4\\
-2xy\\
-\frac{2}{3}u_2+\frac{2}{9}u_4^2\end{array}\right),$$

$$\xi_{5}=\left(\begin{array}{c}
-\frac{1}{3}x^3-\frac{1}{9}u_3x^2+(-\frac{2}{9}u_1+\frac{2}{27}u_3^2)x\\
\frac{1}{3}x^2y+\frac{1}{9}u_3xy\\
-\frac{4}{3}u_2xy+\frac{2}{9}u_1^2-\frac{2}{27}u_1u_3^2\\
-2x^2y^2-\frac{2}{9}u_3u_4xy\\
x^3\!+\!y^3\!+\!u_1x\!+\!u_2y\!+\!u_3x^2\!+\!u_4y^2+\frac{5}{9}u_1u_3-\frac{4}{27}u_3^3\\
-\frac{1}{3}u_3xy\end{array}\right)$$

$$\xi_6=\left(\begin{array}{c}
\frac{1}{3}xy^2+\frac{1}{9}u_4xy\\
-\frac{1}{3}y^3-\frac{2}{9}u_2y-\frac{1}{9}u_4y^2+\frac{2}{27}u_4^2y\\
-2x^2y^2-\frac{2}{9}u_3u_4xy\\
-\frac{4}{3}u_1xy+\frac{2}{9}u_2^2-\frac{2}{27}u_4^2u_2\\
-\frac{1}{3}u_4xy\\
x^3\!+\!y^3\!+\!u_1x\!+\!u_2y\!+\!u_3x^2\!+\!u_4y^2+\frac{5}{9}u_2u_4-\frac{4}{27}u_4^3\end{array}\right)$$
\end{proof}

Therefore, there does not exist $L$ such that $\mathscr M_{6}\subset L(\Lift(F_{3,3}))+\langle L \rangle$ and so
$F_{3,3}$ does not admit a codimension 1 hyperplane section.

\begin{teo}\label{cod2notsimple}
The corank 2 germ $f:(\mathbb R^5,0)\to (\mathbb R^5,0)$ given by
$$f(x,y,u_1,u_2,u_4)=(x^3+y^3+u_1x+u_2y+(-\lambda u_4-u_4^2)x^2+u_4y^2,xy,u_1,u_2,u_4),$$ with $\lambda\neq 0,-1$, has $\mathscr A_e$-codimension 2 and is not simple.
\end{teo}
\begin{proof}
The idea is to use Damon's Theorem relating $\mathscr A_e$-codimension and $_V{\mathscr K}_e$-codimension where $V$ is the discriminant of $F_{3,3}.$ By integrating the linear parts of the vector fields in $\Lift(F_{3,3})$ we obtain the linear parts of diffeomorphisms in $_V\mathscr K$, which are:
\begin{align*}
&\eta_1=(e^{3\alpha}X,e^{2\alpha}Y,e^{2\alpha}U_1,e^{2\alpha}U_2,e^{\alpha}U_3,e^{\alpha}U_4), \\
&\eta_2=(X,Y+\alpha X,U_1,U_2,U_3-4\alpha U_2,U_4-4\alpha U_1), \\
&\eta_3=(X,Y,U_1+\alpha X,U_2,U_3,U_4-2\alpha Y), \\
&\eta_4=(X,Y,U_1;U_2+\alpha X,U_3-2\alpha Y,U_4-2/3\alpha U_2),\\
&\eta_5=(X,Y,U_1,U_2,U_3+\alpha X,U_4),\\
&\eta_6=(X,Y,U_1,U_2,U_3,U_4+\alpha X).
\end{align*}

Let $L:(\mathbb R^6,0)\to \mathbb R$ and suppose $j^1L(X,Y,U_1,U_2,U_3,U_4)=aX+bY+cU_1+dU_2+eU_3+fU_4.$ If $f\neq 0$, by using $\eta_2,\ldots,\eta_6$ we can fix $a=b=c=d=0$, and by using $\eta_1$ we get $j^1L=U_3+\lambda U_4$ where $\lambda$ is a modulus. A complete 2-transversal is given by $U_4^2$ when $\lambda\neq -1$  since $\mathscr M_6^2\subset  T_V\mathscr K_1 L+sp\{U_4^2\}+\mathscr M_6^3$ where $T_V\mathscr K_1L=tL(\Lift_1(F_{3,3}))+L^*\mathscr M_1\mathscr{O}_6$ and $\Lift_1(F_{3,3})$ is the space of vector fields in $\Lift(F_{3,3})$ with zero 1-jet (see \cite{brucewest}). Rescaling we set $L=U_3+\lambda U_4+U_4^2$. If $\lambda\neq 0$, this germ has $_V{\mathscr K}_e$-codimension 2 and is not simple, so the $\mathscr A_e$-codimension of $f$ is 2 and it is not simple.
\end{proof}

By Proposition \ref{equiv}, this is an example of a $\mathscr K$-orbit
which does not admit an open $\A$-orbit.

\section{The extra-nice dimensions}


Mather gave a stratification of the set $J^k(n,p)$ of $k$-jets of smooth mappings by $\mathscr K$-orbits. This induces a partition of $J^k(N,P)$ by $\mathscr K$-orbit bundles. Mather characterized stability in terms of transversality of the $k$-jet extension $j^kf:N\rightarrow J^k(N,P)$ to this stratification.
He showed that there exists a smallest Zariski closed $\mathscr K^k$-invariant set $\Pi^k(n,p)$ in $J^k(n,p)$
such that its complement  in $J^k(n,p)$ is the union of finitely many $\mathscr K^k$-orbits. The codimension of $\Pi^k(n,p)$ decreases as $k$ increases. Moreover, there exists a big enough $k$ for which the codimension of $\Pi^k(n,p)$ attains its minimum. For this $k$ the codimension of the bad set $\Pi(n,p)$ is denoted by $\sigma(n,p.)$  When $ \sigma(n,p) > n,$ then the $k$-jet of a generic map does not meet the set $\Pi(n,p)$ and therefore it is transversal to Mather's stratification in $J^k(N,P)$ and hence it is stable.
He defined the nice dimensions as the pairs $(n,p)$ such that $\sigma(n,p)>n.$ See \cite{duplessiswall} for the notion of semi-nice dimensions, where 2-modality of $\mathscr K$-orbits appears.

As a consequence of Proposition~\ref{cod1simple}, the fact that Mather's bad set has codimension greater than $n$ means that one can detect lack of simplicity at the $\mathscr A_e$-codimension 1 level. Take also into account that $\mathscr K$-orbits of codimension less than or equal to $n$ have stable representatives of $\mathscr A$-codimension less than or equal to $n$ (i.e. there is an open $\A$-orbit in the $\mathscr K$-orbit). If we want to refine this definition to detect lack of simplicity in the $\mathscr A_e$-codimension 2 level we must consider bad sets of codimension greater than $n+1$. Furthermore, since $\mathscr A$-orbits of non stable germs may or may not be open in their $\mathscr K$-orbit, we must consider stratification by $\A$-orbits instead of $\mathscr K$-orbits. This leads to the following

\begin{definition}\label{extranice}
The pair $(n,p)$ is said to be in the extra-nice dimensions if, for large enough $l$, there is a Zariski closed $\A$-invariant subset $\Lambda$ of $J^l(n,p)$, of codimension greater than $n+1$, whose complement is a finite union of $\mathscr A$-orbits.
\end{definition}

It follows from the definition that

\begin{prop}\label{extranicethen2}
If $(n,p)$ is in the extra-nice dimensions then all $\mathscr A_e$-codimension 2 germs are simple.
\end{prop}
\begin{proof}
Suppose we have an $\mathscr A_e$-codimension 2 non simple germ. Then there is a 1-parameter family of $\mathscr A$-orbits of codimension $n+2$, so we cannot find and $\A$-invariant subset $\Lambda$ of $J^l(n,p)$, of codimension greater than $n+1$, whose complement is a finite union of $\mathscr A$-orbits.
\end{proof}

The converse of Proposition~\ref{extranicethen2} is not true as we shall see in a further example (Proposition \ref{nocod2}).

The subset $\Lambda$ in the definition can be constructed containing all non-simple $\mathscr A$-orbits and
all $\mathscr A$-orbits of codimension greater than or equal to $n+2$. In order to be in the extra-nice dimensions, the codimension of this set must be greater than or equal to $n+2$ (see Section \ref{dense}, Proposition \ref{lambda}). It
is not contained and it does not contain Mather's bad set, because
in a $\mathscr K$-orbit of codimension less than or equal to $n$
there can be an infinite number of simple $\mathscr A$-orbits (for example,
augmentations), and in order to have finite $\mathscr A$-orbits in
the complement of our bad set $\Lambda$ we must include in $\Lambda$ some of these $\mathscr
A$-orbits  (i.e. we include the ones of codimension
greater than or equal to $n+2$). In Remark \ref{mathersours} we
compare Mather's bad set and ours for some examples.

The previous definition is well defined because there is an estimate  (depending on $n$ and $p$) for  the degree of
determinacy of $\mathscr A_e$-codimension 1 germs. Namely, results by Mather and Gaffney which can be found in \cite{Wall} and \cite{nunomond} state that if $\A$-cod$(f)=d$ then $\mathscr M_n^{(rp+d)^2}\theta(f)\subset T\mathscr A f$, where r is
the number of branches, and if $\mathscr M_n^{k+1}\theta(f)\subset T\mathscr A f$ then f is $(2k+1)$-$\A$-determined. Combining this we have that $f$ is $(2((rp+d)^2-1)+1)$-$\A$-determined.
So for monogerms we obtain that if $\A_e$-cod$(f)=i$, then $f$ is $(2(p+n+i)
^2-1)$-$\A$-determined, in particular any $\mathscr A_e$-codimension 1 (and therefore $\A$-codimension $n+1$) germ $f:(\R^n,0)\rightarrow (\R^p,0)$ is $(2(p+n+1)^2-1)$-$\mathscr A$-determined.

\begin{prop}\label{extraimplicanice}
If the pair $(n,p)$ is in the extra-nice dimensions, then $(n+1,p+1)$ is in the nice dimensions (in particular, $(n,p)$ is nice dimensions too).
\end{prop}
\begin{proof}
By definition there exists an $\A$-invariant subset $\Lambda$ of
$J^l(n,p)$, of codimension greater than $n+1$, whose complement is a
finite union of $\mathscr A$-orbits. For each one of those $\mathscr
A$-orbits, consider the $\mathscr K$-orbit which contains it. We
therefore have a finite number of $\mathscr K$-orbits of codimension less than or equal to $n+1$ which may
include some $\mathscr A$-orbits which were originally in $\Lambda$.
The complement of this finite number of $\mathscr K$-orbits is
included in $\Lambda$, and therefore the codimension of this
complement is greater than $n+1$. This complement contains all the $\mathscr K$-orbits of codimension greater than or equal to $n+2$ and is Zariski closed. In
conclusion, there exists a $\mathscr K$-invariant subset $\Lambda'$
of codimension greater than $n+1$ such that its complement is a finite
number of $\mathscr K$-orbits. The codimensions of these strata are the same in $(n+1,p+1)$ and so $(n+1,p+1)$ is in the nice
dimensions.
\end{proof}

This means that if $(n+1,p+1)$ is not in the nice dimensions, then $(n,p)$ is not in the extra-nice dimensions.

It is obvious from the definition of nice dimensions that if a pair $(n+1,p+1)$ is in the nice dimensions, then $(n,p)$ is in the nice dimensions, since the codimension of $\mathscr K$-orbits is invariant under unfoldings. However, this is not so obvious for the extra-nice dimensions:

\begin{prop}\label{n+1ton}
If the pair $(n+1,p+1)$ is in the extra-nice dimensions, then $(n,p)$ is in the extra-nice dimensions.
\end{prop}
\begin{proof}
Suppose that $(n,p)$ is not in the extra-nice dimensions. Then,
for any Zariski closed $\mathscr A$-invariant subset $\Gamma$ of $J^l(n,p)$ of
codimension greater than $n+1$, its complement is an infinite number
of $\mathscr A$-orbits.

Let $\Gamma$ be the union of all the $\mathscr A$-orbits of
codimension greater than or equal to $n+2$. If the codimension of
$\Gamma$ is greater than or equal to $n+2$, since it is an $\mathscr
A$-invariant set, its complement is an infinite number of $\mathscr
A$-orbits. On the other hand, the complement of $\Gamma$ is the
union of all $\mathscr A$-orbits of codimension less than or equal
to $n+1$. In the nice dimensions there are a finite number of stable
and $\mathscr A_e$-codimension 1 orbits, therefore we have a
contradiction. Hence, the codimension of $\Gamma$ must be less than or equal to
$n+1$. This means that there is a Zarisky open set in $\Gamma$ foliated by an infinite number of $\mathscr A$-orbits of the same codimension (greater than $n+1$, i.e. $\A_e$-codimension greater than 1).

If these infinite $\mathscr A$-orbits are contained in a single $\mathscr K$-orbit $X$ then $\cod_{J^k(n,p)}X\leq \cod_{J^k(n,p)} \Gamma\leq n+1$. Therefore the $\mathscr K$-orbit has a stable germ $F:(\mathbb K^{n+1},0)\rightarrow (\mathbb K^{p+1},0)$ which is a 1-parameter stable unfolding of the germs of the $\mathscr A$-orbits. By augmenting these germs by $F$ and a Morse function we obtain a stratum of codimension less than or equal to $n+1$ of germs of $\A_e$-codimension greater than 1 in $(n+1,p+1)$, and so $(n+1,p+1)$ is not extra-nice dimensions.

If no single $\mathscr K$-orbit contains the infinite $\A$-orbits, there is a stratum of codimension equal to the codimension of $\Gamma$ with modality of $\mathscr K$-orbits. This implies that $(n+1,p+1)$ cannot be nice dimensions. Therefore, by Proposition \ref{extraimplicanice}, $(n+1,p+1)$ is not extra-nice dimensions.
\end{proof}

\begin{rem}\label{codnonsimple}
The above proof implies that if a pair $(n,p)$ is in the nice
dimensions but not in the extra-nice dimensions, then there exists a
$\mathscr K$-orbit which contains a stratum of codimension less than or equal to $n+1$ of non-simple germs. In fact, the codimension is equal to $n+1$ (see Section \ref{wall}).
\end{rem}

\begin{teo}\label{sectionnice}
If the pair $(n,p)$ is in the extra-nice dimensions then every stable germ $F:(\mathbb K^{n+1},0)\rightarrow(\mathbb K^{p+1},0)$ admits a hyperplane $\mathscr A_e$-codimension 1 section $f:(\mathbb K^n,0)\rightarrow(\mathbb K^p,0).$ The converse is true if $(n+1,p+1)$ is in the nice dimensions.
\end{teo}
\begin{proof}
1) Firstly we are going to prove that if $(n,p)$ is in the extra-nice dimensions then every stable germ $F:(\mathbb K^{n+1},0)\rightarrow(\mathbb K^{p+1},0)$ of $\mathscr K$-codimension $n+1$ admits a hyperplane section of $\mathscr A_e$-codimension 1.

Suppose that there exists a stable  germ $F:(\mathbb K^{n+1},0)\rightarrow(\mathbb K^{p+1},0)$ of $\mathscr K$-codimension $n+1$ which does not admit an $\mathscr A_e$-codimension 1 hyperplane section. Then there exists a $k>1$ and a section   $f':(\mathbb K^n,0)\rightarrow(\mathbb K^p,0)$ of $\mathscr A_e$-codimension $k$  such that all sections of $F$ have $\mathscr A_e$-codimension greater than or equal to $k$. Therefore the $\A$-orbit of $f'$ is not open in its $\mathscr K$-orbit, so the union of a $(k-1)$-parameter family of $\A$-orbits of codimension $n+k$ is a Zariski open subset of the $\mathscr K$-orbit. Then this  $\mathscr K$-orbit is a codimension $n+1$ subset of $J^l(n,p)$ which is not a finite union of $\A$-orbits of minimal codimension, so there does not exist a subset of codimension greater than $n+1$ such that its complement is a finite number of $\A$-orbits, contradicting the fact that $(n,p)$ is in the extra-nice dimensions.

Lets consider now $F:(\mathbb K^{n+1},0)\rightarrow(\mathbb K^{p+1},0)$ a stable germ of $\mathscr K$-codimension $n-r$, $r\geq 0$. Then $F$ is a trivial unfolding of a minimal stable germ $\tilde{F}:(\mathbb K^{n-r},0)\rightarrow(\mathbb K^{p-r},0).$ As $(n,p)$ is in the extra-nice dimensions hence, by Proposition \ref{n+1ton}, $(n-r-1,p-r-1)$ also is and we can apply the above argument to find an $\mathscr A_e$-codimension 1 hyperplane section $\tilde{f}$ of $\tilde{F}.$ Then the augmentation of $\tilde{f}$ is an $\mathscr A_e$-codimension 1 hyperplane section of $F$  (see~\cite{robertamond}).



2) Suppose that $(n,p)$ is not in the extra-nice dimensions. Following the proof of Proposition \ref{n+1ton} there exists $\Gamma$ of codimension less than or equal to $n+1$ with an open Zariski subset foliated by infinite $\A$-orbits of codimension greater than $n+1$.

If no single $\mathscr K$-orbit contains the infinite $\A$-orbits, there is a stratum of codimension equal to the codimension of $\Gamma$ with modality of $\mathscr K$-orbits, contradicting the fact that
that $(n+1,p+1)$ is nice dimensions. Therefore, there exists a $\mathscr K$-orbit which contains the infinite $\A$-orbits. This $\mathscr K$-orbit has codimension less than or equal to $n+1$, therefore there exists a stable germ $F:(\mathbb
K^{n+1},0)\rightarrow(\mathbb K^{p+1},0)$ (not necessarily minimal)
such that all non stable hyperplane sections $f$ of $F$ are not $\A$-simple.
Then by Proposition \ref{cod1simple} $\Aecod(f)>1$.

\end{proof}

In order to determine when a pair $(n,p)$ is in the extra-nice dimensions, by Theorem \ref{sectionnice}, we need to know the stable germs $(\mathbb K^{n+1},0)\rightarrow (\mathbb K^{p+1},0)$, so we describe Mather's procedure to obtain the normal forms for all stable germs. Start with a germ $f_0:(\mathbb K^s,0)\rightarrow (\mathbb K^t,0)$ of rank 0 and find a basis $\{\phi_1,\ldots,\phi_d\}$ of $\frac{\mathscr M_n\theta(f_0)}{T\mathscr K_e f_0}$. Then, $F:(\mathbb K^{s+d},0)\rightarrow (\mathbb K^{t+d},0)$ given by $$F(x,u_1,\ldots,u_d)=(f_0(x)+\sum_{i=1}^d u_i\phi_i(x),u_1,\ldots,u_d)$$ is stable. Furthermore, any stable germ can be obtained by this procedure.

By Proposition \ref{iestrella}, all stable corank 1 germs admit an $\mathscr A_e$-codimension 1 hyperplane section, so we need to study simple germs $f_0$ of corank at least 2. The rank 0 germs $f_0$ have been classified by several authors such as Mather, Arnol'd, Giusti, Wirthm\"{u}ller, Damon, du Plessis, Wall and Gibson. A good account can be found in the book by du Plessis and Wall \cite{duplessiswallbook}.

In the following discussion of simple algebras of corank greater than 1, there are two of corank 2 which play a special role: $$B_{p,q}=(x^p+y^q,xy),\,\,\, B_{p,q}'=(x^p,y^q,xy).$$ Notice that $B_{p,q}'$ is obtained from $B_{p,q}$ by adding its jacobian.

\subsection{Algebras}

We describe the algebras by the number of variables and generators.
We only list the least degenerate ones (which will be enough for our
purpose as we will see in the next section).

Case $n\geq p$:

\begin{enumerate}

\item With 1 generator and any number of variables: $A_k$, $D_k$, $E_6$, $E_7$ and $E_8$, which are of corank 1.

\item With 2 generators and 2 variables: $B_{p,q}$. The stable germs in these algebras are in $(p+q,p+q)$.

\item With 2 generators and $3$ variables: $P_{p,q}=B_{p,q}+(z^2,0)$. The stable germs in these algebras are in $(p+q+2,p+q+1)$.

\item With 2 generators and more than 3 variables there are no simple algebras. For example, the simplest algebra of corank 2 with 4 variables is $(x^2+y^2+z^2,y^2+z^2+\lambda w^2)$ (non-simple because there are 4 quadratics), whose stable germ is in $(9,7)$ and $(8,6)$ is the boundary of the nice dimensions.

\item With more than 2 generators: there are no simple algebras.

\end{enumerate}

Case $n<p$:

\begin{enumerate}

\item With 2 variables: $(B'_{p,q},0,\ldots,0)$ whose stable germs are in $(2p+2q-2+k(p+q-2),2p+2q-1+k(p+q-1))$ and $(B_{p,q},0,\ldots,0)$ whose stable germs are in $(p+q+k(p+q-1),p+q+k(p+q))$, where $k$ is the number of zeros.


\item With 3 variables.

\begin{enumerate}





\item With 4 generators: $f(x,y,z)=(x^2-y^2,y^2+z^2,xz,yz)$. The stable germ is in $(12,13)$. If we add $k$ zeros the stable germ is in $(12+9k,13+10k).$


\item With 5 generators: $f(x,y,z)=(x^2-y^2,y^2+z^2,xy,xz,yz)$ which is obtained from $P_{2,2}$ by adding its jacobian. The stable germ is in $(15,17)$.
If we add $k$ zeros the stable germ is in $(15+12k,17+13k).$


\item With 6 generators: $f(x,y,z)=(x^2,y^2,z^2,xy,xz,yz).$  The stable germ is in $(18,21)$.If we add $k$ zeros the stable germ is in $(18+15k,21+16k).$

\item With more than 6 generators they are obtained from the above ones by adding zeros.
\end{enumerate}

\item With 4 variables the stable germs are in the boundary of the nice dimensions.

\end{enumerate}

\section{The boundary of the extra-nice dimensions}

Many of the calculations which did not make it to the final version of this article were either done or double-checked using an algorithm implemented in the computer package Singular developed by Hernandes, Miranda and Peñafort-Sanchis in \cite{presmatrix}. The calculations which appear in this section have been done by hand unless otherwise stated.

\subsection{The case $n=p$}\label{subsectionnn}

In Theorem \ref{cod2notsimple} we obtain a non-simple $\mathscr A_e$-codimension 2 germs in $(5,5)$ of type $B_{3,3}$. By Proposition \ref{sectionimpliescod2simple} this means that there exists a stable germ in $(6,6)$ which does not admit a hyperplane section of $\mathscr A_e$-codimension 1. Therefore, by Theorem \ref{sectionnice}, $(5,5)$ is not in the extra-nice dimensions. Since $(5,5)$ is in the nice dimensions we have the converse of Theorem \ref{sectionnice}, so to establish the boundary of the extra-nice dimensions we must verify if all stable germs in $(5,5)$ admit hyperplane sections of $\mathscr A_e$-codimension 1. Taking into account the adjacencies of discrete algebra types we only have to investigate the stable germ in $B_{3,2}$:
$$F_{3,2}(x,y,u_1,u_2,u_3)=(x^3+y^2+u_1x+u_2y+u_3x^2,xy,u_1,u_2,u_3)=(X,Y,U_1,U_2,U_3).$$

\begin{lem}\label{32} $\Lift(F_{3,2})$ is generated by
$$\eta_1=\left(\begin{array}{c}6X\\
5Y\\
4U_1\\
3U_2\\
2U_3\end{array}\right), \eta_2=\left(\begin{array}{c}
4U_3Y+2U_1U_2\\
X\\
-5Y-3U_2U_3\\
-3U_1\\
-4U_2\end{array}\right),
\eta_3=\left(\begin{array}{c}
\frac{4}{3}U_2Y\\
-\frac{1}{9} U_3Y\\
X+\frac{1}{9}U_1U_3\\
-\frac{5}{3} Y\\
-\frac{2}{3} U_1+\frac{2}{9} U_3^2\end{array}\right)$$

$$\eta_4=\left(\begin{array}{c}
\frac{3}{2}U_1Y\\
-\frac{1}{4} U_2Y\\
-2U_3Y\\
X+\frac{1}{4}U_2^2\\
-\frac{5}{2}Y\end{array}\right),
\eta_5=\left(\begin{array}{c}
\frac{5}{3}Y^2+\frac{1}{9}U_2U_3Y\\
(-\frac{2}{9}U_1+\frac{2}{27}U_3^2)Y\\
-\frac{4}{3}U_2Y+\frac{2}{9}U_1^2-\frac{2}{27}U_1U_3^2\\
-\frac{2}{9}U_3Y\\
X+\frac{5}{9}U_1U_3-\frac{4}{27}U_3^3\end{array}\right)$$
\end{lem}
\begin{proof} $F_{3,2}$ is a free divisor so $\Lift(F_{3,2})$ is generated by 5 vector fields. The Euler vector field is clearly liftable and the other 4 are all linearly independent and liftable by the following lowerable vector fields:
$$\xi_2=\left(\begin{array}{c}
y+u_2\\
x^2+u_1+u_3x\\
-5xy-3u_2u_3\\
-3u_1\\
-4u_2\end{array}\right),
\xi_3=\left(\begin{array}{c}
-\frac{1}{3}x^2-\frac{1}{9}u_3x\\
\frac{1}{3}xy\\
x^3+y^2+u_1x+u_2y+u_3x^2+\frac{1}{9}u_1u_3\\
-\frac{5}{3}xy\\
-\frac{2}{3}u_1+\frac{2}{9}u_3^2\end{array}\right)$$

$$\xi_4=\left(\begin{array}{c}
\frac{1}{2}xy\\
-\frac{1}{2}y^2-\frac{1}{4}u_2y\\
-2u_3xy\\
x^3+y^2+u_1x+u_2y+u_3x^2+\frac{1}{4}u_2^2\\
-\frac{5}{2}xy\end{array}\right),
\xi_5\left(\begin{array}{c}
-\frac{1}{3}x^3-\frac{1}{9}u_3x^2+(-\frac{2}{9}u_1+\frac{2}{27}u_3^2)x\\
\frac{1}{3}x^2y+\frac{1}{9}u_3xy\\
-\frac{4}{3}u_2xy+\frac{2}{9}u_1^2-\frac{2}{27}u_1u_3^2\\
-\frac{2}{9}u_3xy\\
x^3\!+\!y^2\!+\!u_1x\!+\!u_2y\!+\!u_3x^2\!+\frac{5}{9}u_1u_3-\frac{4}{27}u_3^3
\end{array}\right)$$
\end{proof}

\begin{prop}\label{55}
When $n=p$, $(5,5)$ is the boundary of the extra-nice dimensions.
\end{prop}
\begin{proof}
It remains to prove the existence of a hyperplane section of $\mathscr A_e$-codimension 1 of $F_{3,2}.$
Let $L:\mathbb K^5\rightarrow \mathbb K$ be given by $L(X,Y,U_1,U_2,U_3)=U_3$. Then $\mathscr M_5\subset L(\Lift(F_{3,2}))$ and by Proposition \ref{equiv} iv), $F_{3,2}$ admits a hyperplane section of $\mathscr A_e$-codimension 1, so $(4,4)$ is in the extra-nice dimensions.
\end{proof}

\begin{rem}\label{mathersours}
The pair $(5,5)$ is in the nice dimensions but not in the extra-nice
dimensions. In $J^l(5,5)$ for a sufficiently high $l$ Mather's bad set $\Pi (5,5)$ is a codimension 6 $\mathscr K$-invariant set given by $\Pi (5,5)=\overline{A_6\cup B_{3,3}}.$ Its complement is composed by the $\mathscr K$-orbits $A_i$, $i=1,\ldots,5$, $B_{2,2}$ and $B_{3,2}$, all of which have a stable
representative in $(5,5).$ In the $\mathscr K$-orbit of $B_{3,3}$ there is a 1-parameter family of $\mathscr A$-orbits of
$\mathscr A_e$-codimension 2 (i.e. $\mathscr A$-codimension 7) which
is dense in the $\mathscr K$-orbit (Theorem \ref{cod2notsimple}).
Therefore we cannot find an $\mathscr A$-invariant set of codimension
greater than or equal to $7$  such that the
complement is a finite number of $\mathscr A$-orbits, and hence
$(5,5)$ is not in the extra-nice dimensions.


On the other hand, in $(4,4)$ the set  $\Pi(4,4)$ is a
codimension 5 $\mathscr K$-invariant set in $J^l(4,4)$ given by $\Pi (4,4)=\overline{A_5\cup B_{3,2}}.$ Its complement is composed by the
$\mathscr K$-orbits $A_i$, $i=1,\ldots,4$ and $B_{2,2}$.  Both $\mathscr K$-orbits $A_5$ and $B_{3,2}$ have an open $\mathscr A$-orbit. Let $\Lambda$ be the
union of the closure of the complement of these open orbits in $A_5$ and $B_{3,2}$ and of all the $\mathscr A$-orbits of
codimension greater than or equal to 6 in all the $\mathscr
K$-orbits. This is a codimension 6 $\mathscr A$-invariant subset
such that the complement of it is a finite number of $\mathscr
A$-orbits and hence $(4,4)$ is in the extra-nice dimensions.

\end{rem}

\subsection{How to go from $(n,p)$ ($n\leq p$) to $(n,p+1)$ by adding a $0$
component}\label{bpqtobpq0}

Let $f_0:(\mathbb K^n,0)\longrightarrow (\mathbb K^p,0)$, $n\leq p$, be a $\mathscr K$-finitely determined map-germ. Suppose $\displaystyle \frac{\mathcal{O}_n}{\langle f_0\rangle}\cong\mathbb{K}\{\sigma_0,\ldots,\sigma_r\}$, where $\sigma_0=1$. When $n=p$ we shall consider $\sigma_r=J(f_0)$.

Let $f:(\mathbb K^n,0)\longrightarrow (\mathbb K^{p+1},0)$ be given by $f=(f_0,0)$.

Let $F_0:(\mathbb K^{n+k},0)\longrightarrow (\mathbb K^{p+k},0)$ be a minimal stable unfolding of $f_0$, $F_0(x,u_1,\ldots,u_k)=(\tilde{F}(x,u),u).$

Let $F:(\mathbb K^{n+k+r},0)\longrightarrow (\mathbb K^{p+k+r+1},0)$  be a minimal stable unfolding of $f$ given by $F(x,u,w)=(\tilde{F}(x,u),u, Z(x,w),w)$ where $Z(x,w)=\sum_{i=1}^r\sigma_i(x)w_i.$ We denote the coordinates in target by $(X,U,Z,W).$

We have that

$$dF=\left(\begin{array}{ccccccccc}
\frac{\partial \tilde{F}_1}{\partial x_1} & \cdots & \frac{\partial \tilde{F}_1}{\partial x_n} & \frac{\partial \tilde{F}_1}{\partial u_1} & \cdots & \frac{\partial \tilde{F}_1}{\partial u_k} & 0 & \cdots & 0\\
\vdots & & & & & & & & \\
\frac{\partial \tilde{F}_p}{\partial x_1} & \cdots & \frac{\partial \tilde{F}_p}{\partial x_n} & \frac{\partial \tilde{F}_p}{\partial u_1} & \cdots & \frac{\partial \tilde{F}_p}{\partial u_k} & 0 & \cdots & 0\\
0 & & 0 & 1 & \cdots & 0 & 0 & \cdots & 0\\
\vdots & & & & & & & & \\
0 & & 0 & 0 & \cdots & 1 & 0 & \cdots & 0\\
\frac{\partial Z}{\partial x_1} & \cdots & \frac{\partial Z}{\partial x_n} & 0 & & 0 & \sigma_1 & \cdots & \sigma_r\\
0 & & 0 & 0 & \cdots & 0 & 1 & \cdots & 0\\
\vdots & & & & & & & & \\
0 & & 0 & 0 & \cdots & 0 & 0 & \cdots & 1\end{array}\right)$$

\begin{prop}\label{liftlevantado} Let $\pi_1:(\mathbb K^{p+k+r+1},0)\longrightarrow (\mathbb K^{p+k},0)$ be the natural projection. Then $\Lift(F_0)=\pi_1(\Lift(F)|_{Z=W=0}).$
\end{prop}
\begin{proof}

It is immediate that $\Lift(F_0)\supset\pi_1(\Lift(F)|_{Z=W=0}).$ Only write $dF(\xi)=\eta\circ F$, take $Z=W=0$ and project $\xi$  into the first $n+k$ coordinates and $\eta$ into the first $p+k$ coordinates.

Now let $\xi_0=(\xi_0^1(x,u),\ldots,\xi_0^{n+k}(x,u))\in\theta_{n+k}$ and $\eta_0=(\eta_0^1(X,U),\ldots,\eta_0^{p+k}(X,U))\in\theta_{p+k}$ such that $dF_0(\xi_0)=\eta_0\circ F_0$.


We have that

$$dF(\xi_0,0,\ldots,0)=(\eta_0^1\circ F_0,\ldots,\eta_0^{p+k}\circ F_0,\lambda(x,u,w),0,\ldots,0)$$
where $\lambda(x,u,w)=\sum_{i=1}^n\xi_0^i\frac{\partial Z}{\partial x_i}.$


Since $\displaystyle \frac{\mathcal{O}_{\{(x,u,w)\}}}{\langle F_0,w\rangle}\cong\mathbb{K}\{\sigma_0,\ldots,\sigma_r\}$, it follows from the Preparation Theorem that there exist functions $a_i=a_i(x,u,w)$ such that
$$\lambda(x,u,w)=\sum_{i=0}^r a_i(F_0,w)\sigma_i(x).$$
Now let
$\zeta=(\xi_0^1,\ldots,\xi_0^{n+k},-a_1(F_0,w),\ldots,-a_r(F_0,w))\in\theta_{n+k+r}$ and \linebreak $\eta=(\eta_0^1,\ldots,\eta_0^{p+k},a_0(X,U,W),-a_1(X,U,W),\ldots,-a_r(X,U,W))$. Then, $dF(\zeta)=\eta\circ F$, that is, $\eta\in \Lift(F)$, and $\pi_1(\eta|_{Z=W=0})=\eta_0.$
\end{proof}

\begin{lem} The vectors $\eta_{ZW_j}=Z\sigma_j\frac{\partial}{\partial Z}+Z\frac{\partial}{\partial W_j}\in\theta_{p+k+r+1}$, $j=1,\ldots,r$, and $\eta_{W_iW_j}=W_i\sigma_j\frac{\partial}{\partial Z}+W_i\frac{\partial}{\partial W_j}\in\theta_{p+k+r+1}$, $i,j=1,\ldots,r$, belong to $\Lift(F).$
\end{lem}
\begin{proof}
Write $\xi_{Zw_j}=Z\frac{\partial}{\partial w_j}\in\theta_{n+k+r}$ and $\xi_{w_iw_j}=w_i\frac{\partial}{\partial w_j}\in\theta_{n+k+r}$, $i,j=1,\ldots,r.$ Then $wF(\eta_{ZW_j})=tF(\xi_{Zw_i})$ and $wF(\eta_{W_iW_j})=tF(\xi_{w_iw_j}).$
\end{proof}

\begin{teo}\label{temcomzero} If $F_0$ admits a hyperplane $\mathscr A_e$-codimension 1 section then $F$ also does. The converse is true if the image of the set of linear part of vector fields in $\Lift(F)$ by $\pi_2$ is a subset of $\mathscr{M}_{Z,W}\theta(\pi_2),$ where $\pi_2:(\mathbb K^{p+k+r+1},0)\longrightarrow (\mathbb K^{r+1},0)$ is the projection $\pi_2(X,U,Z,W)=(Z,W)$.
\end{teo}
\begin{proof}
Let $h_0:(\mathbb K^{p+k},0)\to (\mathbb K,0)$ be a linear function such that $h_0=0$ defines  the $\mathscr A_e$-codimension 1 hyperplane section of the discriminant $V_0\subset (\mathbb K^{p+k},0)$ of $F_0,$ that is, $h_0$ has $_{V_0}\mathscr K_e$-codimension 1. Then $th_0(\Lift(F_0))+\langle h_0\rangle\supset \mathscr{M}_{p+k}.$

Let $h:(\mathbb K^{p+k+r+1},0)\to (\mathbb K,0)$ be given by $h(X,U,Z,W)=h_0(X,U)+W_r.$ We are going to prove that $h=0$ defines an $\mathscr A_e$-codimension 1 hyperplane section of the discriminant $V\subset (\mathbb K^{p+k+r+1},0)$ of $F$, that is,
$h$ has $_{V}\mathscr K_e$-codimension 1.

As $F_0$ and $F$ are both minimal stable unfoldings, all liftable vector fields in both cases vanish at zero.

Using notation from the previous lemma, it follows that $\langle th(\eta_{ZW_r}),th(\eta_{W_iW_r})\rangle\supseteq \langle Z,W\rangle_{\O_{{p+k+r+1}}}.$ So $th(\Lift(F))\supseteq \langle Z,W\rangle_{\O_{{p+k+r+1}}}.$

It remains to prove that $X_i,U_j\in th(\Lift(F))+\langle h\rangle$, $i=1,\ldots,n$; $j=1,\ldots,k.$

From hypothesis $X_i,U_j\in th_0(\Lift(F_0))+\langle h_0\rangle$. So, there exists $\eta_i^0(X,U)\in \Lift(F_0)$ such that $X_i=th_0(\eta_i^0)+\alpha h_0,$ $i=1,\ldots,n.$ It follows from Proposition~\ref{liftlevantado} that there exists $\eta_i\in \Lift(F)$, such that $\pi(\eta_i(X,U,0,0))=\eta_i^0(X,U).$ Now, from the definitions of $h$ and $h_0$,  $th(\eta_i)+ \langle Z,W\rangle_{\O_{p+k+r+1}}=th_0(\eta_i^0)+ \langle Z,W\rangle_{\O_{p+k+r+1}}.$ Then, $X_i\in th(\Lift(F))+\langle h\rangle$, $i=1,\ldots,n.$ Similar arguments hold for $U_j$, $j=1,\ldots,k.$ Therefore $h$ has $_{V}\mathscr K_e$-codimension 1.

Conversely, let $h:(\mathbb K^{p+k+r+1},0)\to (\mathbb K,0)$ be a linear function such that $h=0$ defines an $\mathscr A_e$-codimension 1 hyperplane section of the discriminant $V\subset (\mathbb K^{p+k+r+1},0)$ of $F.$ Then
$th(\Lift(F))+\langle h\rangle\supset \mathscr{M}_{p+k+r+1}.$

Since $h$ is linear we can write $h(X,U,Z,W)=h_0(X,U)+h_1(Z,W).$ It follows from the derivative of $F$ above  that $\pi_2(\Lift(F))\cap \mathscr{M}_{X,U}\theta(\pi_2)=\emptyset.$ Therefore $h_0\neq 0$ and $X_i,U_j\in h_0(\Lift(F))+\langle h_0\rangle,$ for all $i=1,\ldots,n$, $j=1,\ldots,k.$ Now from Proposition~\ref{liftlevantado} and as $\pi_2(\Lift(F))\subset \mathscr{M}_{Z,W}\theta(\pi_2),$ we actually have $X_i,U_j\in h_0(\Lift(F_0))+\langle h_0\rangle,$ for all $i=1,\ldots,n$, $j=1,\ldots,k,$ that is, $h_0$ has $_{V_0}\mathscr K_e$-codimension 1.
\end{proof}

\subsection{The case $n=p-1$}

First we consider $B'_{3,3}$ whose stable germ is $F'_{3,3}:(\mathbb K^{10},0)\rightarrow(\mathbb K^{11},0)$ given by $F'_{3,3}(x,y,u_1,u_2,u_3,v_1,v_2,v_3,w_1,w_2)=(x^3+u_1x+u_2y+u_3y^2,y^3+v_1x+v_2y+v_3x^2,xy+w_1x+w_2y,u,v,w).$ In her PhD thesis at Warwick University under the supervision of David Mond, Mirna G\'omez-Morales studied the existence of hyperplane sections of $\mathscr A_e$-codimension 1 for all stable germs in the algebra $B'_{p,q}$ (\cite{tesismirna}). She showed that the stable germ $F'_{3,2}$ of type $B'_{3,2}$ has an $\mathscr A_e$-codimension 1 hyperplane section but $B'_{3,3}$ does not, therefore $(9,10)$ is not in the extra-nice dimensions. Their method is the following. Consider $L:(\K^{11},0)\rightarrow (\K,0)$ be a linear polynomial given by $L(X_1,\ldots,X_{11})=a_1X_1+\ldots+a_{11}X_{11}$. Given an immersion $i:(\K^{10},0)\rightarrow (\K^{11},0)$, such that $L\circ i=0$, by Damon's Theorem we have that $\mathscr A_e$-$\cod(i^*(F'_{3,3}))=\mathscr
K_{V_e}$-$\cod(i)= _V\mathscr K_e$-$\cod(L)$, where $V$ is the discriminant of $F'_{3,3}$. Let $M_{F'_{3,3}}(X_1,\ldots,X_{11})$ be the matrix whose entries correspond to the linear parts in $X_1,\ldots,X_{11}$ of the generators of $tL(\Derlog(\Delta(F'_{3,3})))+\langle L\rangle\theta(L)$. Then $$M_{F'_{3,3}}(X_1,\ldots,X_{11})=(X_1,\ldots,X_{11})\cdot N_{F'_{3,3}}(a_1,\ldots,a_{11}).$$ Therefore, $\mathscr A_e$-$\cod(i^*(F'_{3,3}))=1$ if and only if the rank of $N_{F'_{3,3}}$ is $11$. In fact, the same argument proves that $\mathscr A_e$-$\cod(i^*(F'_{3,3}))=k$ if and only if the rank of $N_{F'_{3,3}}$ is $12-k$ for $k\geq 1$.

\begin{lem}
The linear parts of the generators of $\Lift(F'_{3,3})$ are

\begin{tabular}{rcl}
$\eta_1$&$=$&$3X\frac{\partial}{\partial X}+3Y\frac{\partial}{\partial Y}+2Z\frac{\partial}{\partial Z}+2U_1\frac{\partial}{\partial U_1}+2U_2\frac{\partial}{\partial U_2}+U_3\frac{\partial}{\partial U_3}+2V_1\frac{\partial}{\partial V_1}$\\&&$+2V_2\frac{\partial}{\partial V_2}+V_3\frac{\partial}{\partial V_3}+W_1\frac{\partial}{\partial W_1}+W_2\frac{\partial}{\partial W_2}$,\\
$\eta_2$&$=$&$3X\frac{\partial}{\partial Z}-3V_2\frac{\partial}{\partial V_3}+3U_1\frac{\partial}{\partial W_1}+4U_2\frac{\partial}{\partial W_2}$,\\
$\eta_3$&$=$&$3Y\frac{\partial}{\partial Z}-3U_1\frac{\partial}{\partial U_3}+4V_1\frac{\partial}{\partial W_1}+3V_2\frac{\partial}{\partial W_2}$,\\
$\eta_4$&$=$&$9X\frac{\partial}{\partial U_1}+3V_1\frac{\partial}{\partial V_3}+3Z\frac{\partial}{\partial W_1}+2U_1\frac{\partial}{\partial W_2}$,\\
$\eta_5$&$=$&$3Y\frac{\partial}{\partial U_1}-3Z\frac{\partial}{\partial U_3}+V_1\frac{\partial}{\partial W_2}$,\\
$\eta_6$&$=$&$3X\frac{\partial}{\partial U_2}-3U_2\frac{\partial}{\partial U_3}+Z\frac{\partial}{\partial W_2}$,\\
$\eta_7$&$=$&$X\frac{\partial}{\partial U_3}$,\\
$\eta_8$&$=$&$3U_2\frac{\partial}{\partial U_3}+9Y\frac{\partial}{\partial V_2}+2V_2\frac{\partial}{\partial W_1}+3Z\frac{\partial}{\partial W_2}$,\\
$\eta_9$&$=$&$3Y\frac{\partial}{\partial V_1}-3V_1\frac{\partial}{\partial V_3}+Z\frac{\partial}{\partial W_1}$,\\
$\eta_{10}$&$=$&$3X\frac{\partial}{\partial V_2}-3Z\frac{\partial}{\partial V_3}+U_2\frac{\partial}{\partial W_1}$,\\
$\eta_{11}$&$=$&$3Y\frac{\partial}{\partial V_3}+X\frac{\partial}{\partial W_2}$,\\
$\eta_{12}$&$=$&$X\frac{\partial}{\partial W_1}$,
$\eta_{13}=Y\frac{\partial}{\partial W_1}$,
$\eta_{14}=Y\frac{\partial}{\partial W_2}$,
\end{tabular}

where $(X,Y,Z,U_1,U_2,U_3,V_1,V_2,V_3,W_1,W_2)$ are the coordinates of the target.
\end{lem}
\begin{proof}
$\eta_1$ is the Euler vector field and lowerable vector fields for the remaining ones are

\begin{tabular}{l}
$\xi_2=-u_2\frac{\partial}{\partial x}+x^2\frac{\partial}{\partial y}-3v_2\frac{\partial}{\partial v_3}+3u_1\frac{\partial}{\partial w_1}+4u_2\frac{\partial}{\partial w_2}$,\\
$\xi_3=3y^2\frac{\partial}{\partial x}-v_1\frac{\partial}{\partial y}-3u_1\frac{\partial}{\partial u_3}+4v_1\frac{\partial}{\partial w_1}+3v_2\frac{\partial}{\partial w_2}$,\\
$\xi_4=(-3x^2-2u_1)\frac{\partial}{\partial x}+X\frac{\partial}{\partial u_1}+3v_1\frac{\partial}{\partial v_3}+3Z\frac{\partial}{\partial w_1}+2u_1\frac{\partial}{\partial w_2}$,\\
$\xi_5=-v_1\frac{\partial}{\partial x}+3Y\frac{\partial}{\partial u_1}-3Z\frac{\partial}{\partial u_3}+v_1\frac{\partial}{\partial w_1}$,\\
$\xi_6=-xy\frac{\partial}{\partial x}+3X\frac{\partial}{\partial u_2}-3u_2\frac{\partial}{\partial u_3}+z\frac{\partial}{\partial w_2}$,\\
$\xi_7=X\frac{\partial}{\partial u_3}$,\\
$\xi_8=(-3y^2-2v_2)\frac{\partial}{\partial y}+3u_2\frac{\partial}{\partial u_3}+9Y\frac{\partial}{\partial v_2}+2v_2\frac{\partial}{\partial w_1}+3Z\frac{\partial}{\partial w_2}$,\\
$\xi_9=(-v_1x-xy)\frac{\partial}{\partial y}+3Y\frac{\partial}{\partial v_1}-3v_1\frac{\partial}{\partial v_3}+Z\frac{\partial}{\partial w_1}$,\\
$\xi_{10}=-u_2\frac{\partial}{\partial y}+3X\frac{\partial}{\partial v_2}-3Z\frac{\partial}{\partial v_3}+u_2\frac{\partial}{\partial w_1}$,\\
$\xi_{11}=-x^2y\frac{\partial}{\partial y}+3Y\frac{\partial}{\partial v_3}+X\frac{\partial}{\partial w_2}$,\\
$\xi_{12}=X\frac{\partial}{\partial w_1}$, $\xi_{13}=Y\frac{\partial}{\partial w_1}$, $\xi_{14}=Y\frac{\partial}{\partial w_2}$,\\
\end{tabular}

where $X=x^3+u_1x+u_2y+u_3y^2$, $Y=y^3+v_1x+v_2y+v_3x^2$ and $Z=xy+w_1x+w_2y.$

\end{proof}

The next result shows a counterexample of the converse of Proposition \ref{extranicethen2}.
\begin{prop}\label{nocod2}
Any hyperplane section of $F'_{3,3}$ has $\mathscr A_e$-codimension greater than 2.

\end{prop}
\begin{proof}
Analyzing the linear parts of $\Lift(F'_{3,3})$ we can see that the rank of $N_{F'_{3,3}}$ is 8 so the best possible hyperplane section has $\mathscr A_e$-codimension 4.
\end{proof}

This means that in $(9,10)$ all $\mathscr A_e$-codimension 2 germs are simple, but $(9,10)$ is not in the extra-nice dimensions.

\begin{prop}
When $n=p-1$, $(9,10)$ is the boundary of the extra-nice dimensions.
\end{prop}
\begin{proof}
We must analyze corank 2 stable germs in $(9,10)$. These are given by the algebra type $(B_{3,2},0)$. By Lemma \ref{32}, $F_{3,2}$ admits an $\mathscr A_e$-codimension 1 hyperplane section, and by Theorem \ref{temcomzero} the stable germ for $(B_{3,2},0)$ admits a section too. Since $(9,10)$ is in the nice dimensions, by the converse of Theorem \ref{sectionnice}, $(8,9)$ is in the extra-nice dimensions.
\end{proof}

\subsection{The case $n<p-1$}

Up to now we have at the boundary of the extra-nice dimensions (5,5)
and (9,10). These two pairs of dimensions lie in the line of
equation $5n-4p-5=0$ in the $(n,p)$-plane. In fact,
we can generalise this to include the case $n<p-1$:

\begin{prop}
If $n\leq p$ the boundary of the extra-nice dimensions is given by
$5n-4p-5=0$, $p\geq 5$.
\end{prop}
\begin{proof}
Let $F_0:(\mathbb K^{10},0)\to (\mathbb K^{11},0)$ be the minimal stable germ in $B'_{3,3}$ algebra and $F:(\mathbb K^{10+4k},0)\to (\mathbb K^{11+5k},0)$ be the minimal stable germ in $(B'_{3,3},0,\ldots,0)$ where $k$ is the number of zeros. Following notation of section~\ref{bpqtobpq0} we have $\sigma_1=x,$ $\sigma_2=y,$ $\sigma_3=x^2,$ $\sigma_4=y^2$ and $Z(x,y,w_1,w_2,w_3,w_4)=w_1x+w_2y+w_3x^2+w_4y^2$. Therefore, $\partial Z/\partial x=w_1+2w_3x$ and $\partial Z/\partial y=w_2+2w_4y$. So one can see that $F$ satisfies the conditions of Theorem~\ref{temcomzero}.
As $F_0$ does not admit a hyperplane section of $\mathscr A_e$-codimension 1 then by Theorem~\ref{temcomzero} $F$ does not. If $F_0:(\mathbb K^6,0)\to (\mathbb K^6,0)$ is the minimal stable germ in $B_{3,3}$ algebra and $F:(\mathbb K^{6+5k},0)\to (\mathbb K^{6+6k},0)$ is the minimal stable germ in $(B_{3,3},0,\ldots,0)$ where $k$ is the number of zeros, then similarly $F$ does not admit a hyperplane section of $\mathscr A_e$-codimension 1 since $\sigma_1=x,$ $\sigma_2=y,$ $\sigma_3=x^2,$ $\sigma_4=y^2.$

Let now $F_0:(\mathbb K^8,0)\to (\mathbb K^9,0)$ be the minimal stable germ in $B'_{3,2}$ algebra and $F:(\mathbb K^{8+3k},0)\to (\mathbb K^{9+4k},0)$ be the minimal stable germ in $(B'_{3,2},0,\ldots,0)$ where $k$ is the number of zeros. We have  $Z(x,y,w_1,w_2,w_3)=w_1x+w_2y+w_3x^2$. Therefore, $\partial Z/\partial x=w_1+2w_3x$ and $\partial Z/\partial y=w_2$. So one can see that $F$ satisfies the conditions of Theorem~\ref{temcomzero}.
As $F_0$  admits a hyperplane section of $\mathscr A_e$-codimension 1 then by Theorem~\ref{temcomzero} $F$ does. Similarly if $F_0:(\mathbb K^5,0)\to (\mathbb K^5,0)$ is the minimal stable germ in $B_{3,2}$ algebra and $F:(\mathbb K^{5+4k},0)\to (\mathbb K^{5+5k},0)$ is the minimal stable germ in $(B_{3,2},0,\ldots,0)$ where $k$ is the number of zeros,  then similarly $F$  admits a hyperplane section of $\mathscr A_e$-codimension 1.

Suppose $p=n+\ell$.  We have:

\begin{tabular}{||c|c|c|c||}

\hline
 \text{algebra} & $k$ & \text{stable germ} & \text{hyperplane section}\\
\hline
$(B_{3,3},0,\ldots,0)$ & $k=\ell$ & $(6+5\ell,6+6\ell)$ & $(5+5\ell,5+6\ell)$\\
\hline
$(B'_{3,3},0,\ldots,0)$ & $k=\ell -1$ & $(6+4\ell,6+5\ell)$ & $(5+4\ell,5+5\ell)$\\
\hline
$(B_{3,2},0,\ldots,0)$ & $k=\ell$ & $(5+4\ell,5+5\ell)$ & $(4+4\ell,4+5\ell)$\\
\hline
$(B'_{3,2},0,\ldots,0)$ & $k=\ell -1$ & $(5+3\ell,5+4\ell)$ & $(4+3\ell,4+4\ell)$\\
\hline
\end{tabular}

\vspace{0.5cm}

In conclusion, all the pairs of dimensions of sections of stable germs corresponding to algebras $B_{p,q}$ or $B_{p,q}'$ (or these two with zeros added) with $p+q\leq 5$ lie in the
extra-nice dimensions. The ones corresponding to
$B_{p,q}$ or $B_{p,q}'$ (or these two with zeros added) with $p+q\geq 6$ are not in the extra-nice dimensions.

Therefore $(5+4\ell,5+5\ell)$ is in the boundary of the extra-nice dimensions and belongs to the line $5n-4p-5=0.$ The sections of stable germs corresponding to algebras with 3 or more variables lie to the right of this line and are not in the extra-nice dimensions.

\end{proof}

\subsection{The case $n=p+1$}

Consider the simplest corank 2 algebra in this setting and its stable unfolding $P_{2,2}:(\mathbb K^6,0)\to (\mathbb K^5,0)$ given by
$$P_{2,2}(x,y,z,u_1,u_2,u_3)=(x^2+y^2+z^2+u_1x+u_2y,xy+u_3z,u_1,u_2,u_3).$$

\begin{lem}\label{p22} The generators of $\Lift(P_{2,2})$ are

$$\eta_1=\left(\begin{array}{c}2X\\
2Y\\
U_1\\
U_2\\
U_3\end{array}\right), \eta_2=\left(\begin{array}{c}
6YU_1-2U_2U_3^2\\
-YU_2+2U_1U_3^2\\
-8Y\\
4X+U_2^2+4U_3^2\\
-U_2U_3\end{array}\right),
\eta_3=\left(\begin{array}{c}
4YU_3+3U_1U_2U_3\\
2XU_3+2U_3^2\\
-4U_2U_3\\
-4U_1U_3\\
2Y\end{array}\right)$$

$$\eta_{4,5}=\left(\begin{array}{c}
6YU_2-2U_1U_3^2\\
-YU_1+2U_2U_3^2\\
4X+U_1^2+4U_3^2\\
-8Y\\
-U_1U_3\end{array}\right),
\left(\begin{array}{c}
4(X^2-12Y^2-6YU_1U_2+U_3^2(5X+3U_1^2+3U_2^2))\\
-8XY-16YU_3^2-9U_1U_2U_3^2\\
2XU_1+36YU_2-14U_1U_3^2\\
36YU_1+2XU_2-14U_1U_3^2\\
-10XU_3-2U_3^3\end{array}\right)$$

\end{lem}
\begin{proof} The discriminant (image of critical points set) $V$ of $P_{2,2}$ is a free divisor so $\Derlog(V)$ is generated by 5 vector fields. In order to obtain them we calculate the order 5 minors of the differential and eliminate the variables $x,y,z$ from the ideal generated by $2x^2-2y^2+u_1x-u_2y, u_1u_3+2U_3x-2yz, u_2u_3-2u_3y-2xz, X-(x^2+y^2+z^2+u_1x+u_2y), Y-(xy+u_3z)$ (the first three are the minors). Using the computer package Singular we obtain the defining equation of $V.$ Again using Singular we can compute with syzygies the generators of $\Derlog(V)$. It can be seen that all of them are linearly independent and liftable.
\end{proof}

Notice that there are only 4 vector fields with non zero linear parts.

\begin{prop}
When $n=p+1$, $(5,4)$ is the boundary of the extra-nice dimensions.
\end{prop}
\begin{proof}
Analyzing the vector fields in Lemma \ref{p22} and by Proposition \ref{equiv} iv), $P_{2,2}$ does not admit a hyperplane section of $\mathscr A_e$-codimension 1, so $(5,4)$ is not in the extra-nice dimensions. In fact, it does not admit an $\mathscr A_e$-codimension 2 section either. Since $P_{2,2}$ is the simplest corank 2 algebra in $(n+1,n)$ and $(5,4)$ is in the nice dimensions, by the converse of Theorem \ref{sectionnice}, $(4,3)$ is in the extra-nice dimensions and so $(5,4)$ is the boundary of the extra-nice dimensions.
\end{proof}

\subsection{The case $n>p+1$}

From the discussion of the algebras, we know that there are no rank 0 simple algebras in this case. Here the boundary of the nice dimensions is given by the non simple $\A_e$-codimension 1 section of the stable unfolding of the simplest rank 0 germ for each case $(n+k,n)$, $k>2$.

In the case $(n+2,n)$ this algebra is $(x^2+y^2+z^2,y^2+\lambda z^2+w^2)$, whose stable unfolding is in $(9,7)$ and the boundary of the nice dimensions is $(8,6)$. By Proposition \ref{extraimplicanice} $(7,5)$ is not in the extra-nice dimensions. Since there are no stable corank 2 germs in $(7,5)$ (or below) the boundary of the extra-nice dimensions is $(7,5)$.

On the other hand, for $(n+k,n)$ with $k>2$, the boundary of the nice dimensions is given by the non simple $\A_e$-codimension 1 section of the stable unfolding of $x^3+y^3+z^3+\lambda xyz+\Sigma_{i=1}^k w_i^2$, whose stable unfolding is in $(10+k,8)$, and so this boundary is $(9+k,7)$. By Proposition \ref{extraimplicanice} $(8+k,6)$ is not in the extra-nice dimensions. Since there are no stable corank 2 germs in $(8+k,6)$ (or below) the boundary of the extra-nice dimensions is $(8+k,6)$. We remark that in $(8+k,6)$ there may be a non simple $\mathscr A_e$-codimension 2 corank 1 germ obtained as a section (by a codimension 2 plane) of the stable unfolding of the germ $x^3+y^3+z^3+\lambda xyz+\Sigma_{i=1}^k w_i^2$.

\subsection{Diagram of the boundary}

From all the above discussions we can draw the boundary of the extra-nice dimensions and compare it to the boundary of the nice dimensions. In Figure \ref{table} the dotted line represents the boundary of the nice dimensions and the continuous line is the boundary of the extra-nice dimensions.
\begin{figure}[!htb]
\centering
\includegraphics[width=1\linewidth]{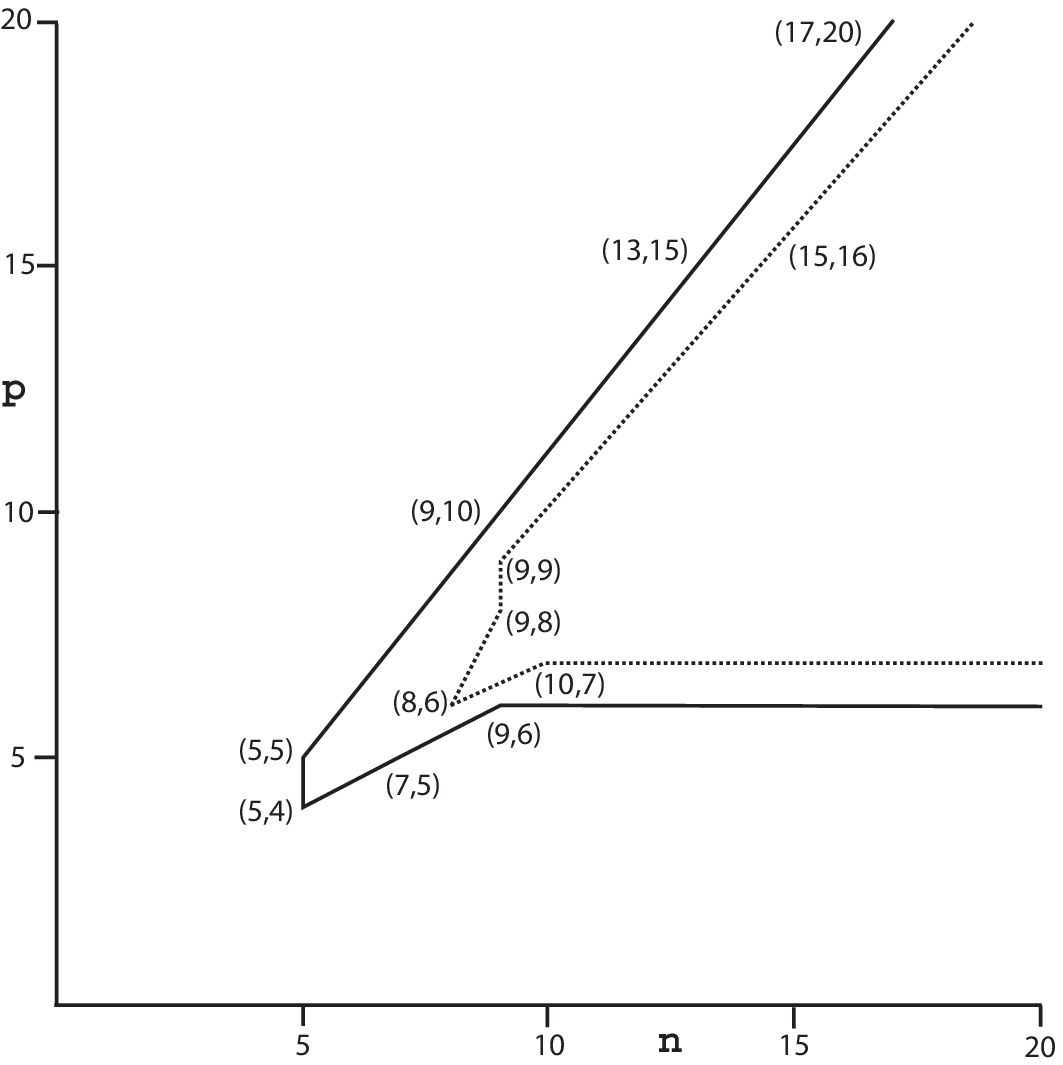}
\caption{} \label{table}
\end{figure}

\section{Locally stable 1-parameter families are dense in the extra-nice dimensions}\label{dense}

In analogy to Mather's characterization of the nice dimensions as those where stable maps are dense, we characterize the extra-nice dimensions as those where stable 1-parameter families are dense. Let $N$ be a compact manifold.

\begin{definition}
Let $F: N\times [0,1] \to P$ be a family such that $(F,t):N\times [0,1] \to P\times [0,1]$ is a stable map. Then $F$ is a {\it locally stable 1-parameter family } if $F_t: N \to P$ is a stable map for all $t \in [0,1]$ except for possibly
a finite number of values $\{t_1,\ldots, t_k\}$ and the non-stable singularities of $F_{t_i}$ are a finite number of points $x_j \in N,$ such that the
map $(F, t) : N\times [0,1] \to P\times [0,1]$ is a locally $\mathcal A_e$-versal unfolding of $F_{t_i}$ for all non-stable points $x_j.$
\end{definition}

From this definition it follows that $F_{t_i}:(N,x_j)\rightarrow(P,F_{t_i}(x_j))$ has $\mathscr A_e$-codimension 1. Locally stable 1-parameter families are known as pseudo-isotopies in \cite{cerf} or \cite{igusa}.



Suppose that $(n,p)$ is in the nice dimensions. We define a  stratification of $J^k(n,p)$ by $\A$-orbits. Let $\Lambda(n,p)=\{\sigma\in J^k(n,p)/$ $\A^k$-codimension of $\sigma \geq n+2\}.$ $\Lambda(n,p)$ is an algebraic set, hence it admits an $\A^k$-invariant stratification. The set $J^k(n,p)\backslash \Lambda(n,p)$ has a finite stratification by $\A$-orbits of  codimension less than or equal to $n+1$ (i.e. orbits of germs of $\mathscr A_e$-codimension less than or equal to 1). Since this stratification is $\A$-invariant it induces an stratification  $S(N,P)$ of $J^k(N,P)$.

\begin{lem}\label{lambda}
Let $(n,p)$ be in the nice dimensions. $(n,p)$ is in the extra-nice dimensions if and only if the set $\Lambda(n,p)$ has codimension bigger than or equal to $n+2.$
\end{lem}
\begin{proof}
If the codimension of $\Lambda(n,p)$ is less than or equal to $n+1$, then it is foliated by infinite $\A$-orbits and you cannot find a subset of codimension greater than $n+1$ such that its complement is a finite number of $\A$-orbits. The ``if" part is trivial by definition since in the nice dimensions there are a finite number of $\A$-orbits of codimension less than or equal to $n+1$.
\end{proof}

Consider a 1-parameter family $F: N\times [0,1] \to P$ and consider the partial jet extension $$j_1^kF:N\times[0,1]\rightarrow J^k(N,P)$$ given by $j_1^kF(x,t)=j^kF_t(x)$.

\begin{lem}\label{prevteo}
If $F:N\times [0,1] \to P$ is a locally stable 1-parameter family then $j_1^kF$ is transverse to the stratification $S(N,P)$. The converse holds if $(n,p)$ is in the extra-nice dimensions.
\end{lem}
\begin{proof}
First of all, the Versality Theorem 3.3 in \cite{Wall} (or 4.1.4 in \cite{nunomond}) says that an unfolding $(F,t)$ of $f$ is locally $\A_e$-versal if and only if $T\A_e(f)+Sp\{\frac{\partial F}{\partial t}\}=\theta(f)$. Based on this it can be seen (Theorem 4.1.11 in \cite{nunomond} or Proposition 2.2 in \cite{walltransversality} for $\mathscr K$-equivalence) that $(F,t)$ is locally $\A_e$-versal if and only if $j_1^kF$ is transversal to the $\A$-orbit of $f$.

If $F$ is a locally stable 1-parameter family, by definition $(F,t)$ is an $\A_e$-versal unfolding of $F_{t_i}$ at the points where it has non stable singularities (and these are of $\A_e$-codimension 1 only). By the above result this means that $j_1^kF$ is transversal to all $\mathscr A$-orbits of codimension less than or equal to $n+1$. Since $F$ does not have singularities of $\A_e$-codimension greater than 1 for any $t$, $j_1^kF$ is also transversal to $\Lambda(N,P)$. Therefore, $j_1^kF$ is transversal to the stratification $S(N,P)$.

Conversely, if $(n,p)$ is in the extra-nice dimensions, $\Lambda(N,P)$ has codimension greater than or equal to $n+2$, so transversality to the stratification $S(N,P)$ means that $j_1^kF$ is transversal to all $\A$-orbits of codimension less than or equal to $n+1$ and avoids $\Lambda(N,P)$. Therefore, $(F,t)$ unfolds versally all the $F_{t_i}$ which have non stable singularities of $\A_e$-codimension 1. Since $N$ is compact the points $x_j$ where $F_{t_i}$ is non stable are finite, so $F$ is a locally stable 1-parameter family.
\end{proof}

\begin{teo}\label{teodense}
Let $N$ and $P$ be manifolds of dimension $n$ and $p$, with $(n+1,p+1)$ nice dimensions. The subset of locally stable 1-parameter families in $C^{\infty}(N\times [0,1],P)$ is dense if and only if $(n,p)$ is in the extra-nice dimensions.
\end{teo}
\begin{proof}
The pair $(n,p)$ is in the extra-nice dimensions if, by definition, there exists a bad set $\Lambda(n,p)\subset J^k(n,p)$ of codimension greater than $n+1$ such that its complement is a finite number of $\mathscr A$-orbits. This set induces a bad set $\Lambda(N,P)$ in $J^k(N,P)$ of codimension greater than $n+1$. Thus, for a generic family $F\in C^{\infty}(N\times [0,1],P)$, $j_1^kF(N\times [0,1])\cap \Lambda(N,P)=\emptyset$. By Thom's transversality theorem, the set of families $F$ such that $j^k_1F$ is transversal to any $\A$-orbit of codimension less than or equal to $n+1$ is also a residual set. Since there are finite $\A$-orbits of codimension less than or equal to $n+1$, the set of families $F$ such that their partial jet extension is transverse to $S(N,P)$ is a finite intersection of residual sets and so is residual. Equivalently, by Lemma \ref{prevteo}, the set of locally stable 1-parameter families in $C^{\infty}(N\times [0,1],P)$ is dense.

Now suppose $(n,p)$ is not in the extra-nice dimensions. Then the codimension of $\Lambda(N,P)$ is less than or equal to $n+1$. In Section 5, for any $(n,p)$ in the boundary of the extra-nice dimensions such that $(n+1,p+1)$ is nice dimensions, we construct a stable germ $\widetilde{F}$ in $(n+1,p+1)$ such that any section of it $f_0$ is not simple and such that $\widetilde{F}=(F,t)$ is a 1-parameter stable unfolding of $f_0$. By taking a trivial unfolding of these stable germs we obtain stable germs with this property in any pair of dimensions outside the extra-nice dimensions. Considering one of the sections of these stable germs in $(n,p)$ and its deformation $F$, $F$ is clearly not a locally stable 1-parameter family but $j_1^kF(N\times[0,1])\cap\Lambda(N,P)\neq \emptyset$ and $j_1^kF$ is transversal to $S(N,P)$. There is a sufficiently small neighbourhood of $F$ in $C^{\infty}(N\times[0,1],P)$ such that any $G$ in that neighbourhood satisfies that $j_1^kG$ is transversal to $S(N,P)$ and $j_1^kG(N\times[0,1])\cap\Lambda(N,P)\neq \emptyset$. Therefore, there is an open set of non locally stable 1-parameter families and so the set of locally stable 1-parameter families is not dense.
\end{proof}

A  summary of our main results in Theorems \ref{sectionnice} and \ref{teodense} we have the following

\begin{coro} Let $(n+1,p+1)$ be in the nice dimensions. The following are equivalent
\begin{enumerate}
\item[i)] $(n,p)$ is in the extra-nice dimensions,
\item[ii)] $cod_{J^k(n,p)}\Lambda(n,p)\geq n+2$,
\item[iii)] the subset of locally stable 1-parameter families in $C^{\infty}(N\times [0,1],P)$ is dense,
\item[iv)] every stable germ $F:(\mathbb K^{n+1},0)\rightarrow(\mathbb K^{p+1},0)$ admits a hyperplane $\mathscr A_e$-codimension 1 section $f:(\mathbb K^n,0)\rightarrow(\mathbb K^p,0).$
\end{enumerate}
\end{coro}

\section{Codimension of non-simple germs}\label{wall}

In this section we answer partially a question posed by Wall to the first author during his talk at the workshop on "Singularities in Generic Geometry and Applications IV" held in Edinburgh in 2013: what is the codimension of the non-simple germs?

Let $NS$ denote the $\mathscr A$-invariant subset of $J^l(n,p)$
composed of all non $\mathscr A$-simple orbits. If $(n,p)$ is in the
nice dimensions, by Proposition \ref{cod1simple}, all $\mathscr
A_e$-codimension 1 germs are simple, so if a germ is not simple its
$\mathscr A_e$-codimension is at least 2 ($\mathscr A$-codimension
$n+2$). Therefore $\cod_{J^l(n,p)}(NS)\geq n+1$. Similarly, from the
definition of the extra-nice dimensions, if $(n,p)$ is in the extra
nice dimensions, then $\cod_{J^l(n,p)}(NS)\geq n+2$. In fact, if
$(n,p)$ is in the nice dimensions but not in the extra-nice
dimensions, Remark \ref{codnonsimple} shows that
$\cod_{J^l(n,p)}(NS)\leq n+1$, so $\cod_{J^l(n,p)}(NS)=n+1$.

This naturally leads to the following definition

\begin{definition}
The pair $(n,p)$ is said to be in the $\Delta_m$-nice dimensions if, for large enough $k$, there is an $\A$-invariant subset $\Lambda$ of $J^l(n,p)$, of codimension greater than $n+m$, whose complement is a finite union of $\mathscr A$-orbits.
\end{definition}

$\Delta_1$-nice dimensions are the extra-nice dimensions and $\Delta_0$-nice dimensions are the nice dimensions. With this definition, if $(n,p)$ is in the $\Delta_m$-nice dimensions but not in the $\Delta_{m+1}$-nice dimensions, then $\cod_{J^l(n,p)}(NS)=n+m+1$.

\begin{ex}
In the case $n=p$, the Thom-Levine example in the introduction is an example in the boundary of the $\Delta_0$-nice dimensions of an $\A_e$-codimension 1 germ which is not simple. It has corank 3. Inside the $\Delta_0$-nice dimensions all $\A_e$-codimension 1 germs are simple.

From Theorem \ref{cod2notsimple} we have an example in the boundary of the $\Delta_1$-nice dimensions of an $\A_e$-codimension 2 germ which is not simple. It has corank 2. Inside the $\Delta_1$-nice dimensions all $\A_e$-codimension 2 germs are simple.

From \cite{riegerruasmdef} and \cite{nishi} we know that if $n\geq 3$, then a germ $f:(\K^n,0)\rightarrow (\K^n,0)$ with $m_0(f)\geq n+3$ is not simple. In fact, in $(3,3)$ we find the germ $(x,y,z^6+yz^2+xz)$ which has $\A_e$-codimension 3 and is not simple with 1 modal parameter. It has corank 1. From \cite{rieger} and \cite{riegerunimodal} we know that all $\A_e$-codimension 3 germs in $(2,2)$ are simple. This means that $(3,3)$ is the boundary of the $\Delta_2$-nice dimensions. Inside the $\Delta_2$-nice dimensions all $\A_e$-codimension 3 germs are simple.

Finally, in $(2,2)$ there is an $\A_e$-codimension 4 germ which is not simple. It can be found in \cite{rieger}, $(x,xy+y^6\pm y^9+\alpha y^9)$. Since in $(1,1)$ all germs are simple, this means that $(2,2)$ is the boundary of the $\Delta_3$-nice dimensions. In fact, $(2,2)$ is the boundary of the $\Delta_m$-nice dimensions for $m\geq 3$!

So we get a stratification of the $n=p$ dimensions by $(9,9)\supset(5,5)\supset(3,3)\supset(2,2)$.
\end{ex}


\end{document}